\documentclass[a4paper,10pt]{article}
\usepackage[utf8]{inputenc}

\usepackage{wrapfig}
\usepackage{amsmath} % numeraci— de teoremes subetidata a les seccions, per exemple
\usepackage{amsthm} % Altres imprescindibles, com environement proof.
\usepackage{amssymb} % carˆcters matemˆtics, imprescindible
\usepackage{enumerate} % enumerar amb coses diferents a nombres
\usepackage{esint} %integrals
\usepackage{pgf,tikz} %tikz: dibuixos i esquemes, com el cercle
\usetikzlibrary{arrows} %fletxes al tikz
 \usepackage{yfonts} % textgoth, textswab, textfrak
 \usepackage{mathrsfs} % mathscr
 \usepackage{graphicx}
\usepackage{caption}
\usepackage{subcaption}
\usepackage{mathtools} %alinear matrius
\usepackage[titletoc,toc]{appendix} % appendix - http://tex.stackexchange.com/questions/24750/article-appendix-with-sections-and-toc-entries-in-the-form-appendix-a

\definecolor{ffffff}{rgb}{1.0,1.0,1.0}
\definecolor{qqqqff}{rgb}{0.0,0.0,1.0}
\definecolor{ffqqqq}{rgb}{1.0,0.0,0.0}
\definecolor{zzzzqq}{rgb}{0.6,0.6,0.0}

\newcommand*\circled[1]{\tikz[baseline=(char.base)]{
            \node[shape=circle,draw,inner sep=2pt] (char) {#1};}}

\newcommand*\squared[1]{\tikz[baseline=(char.base)]{
            \node[shape=rectangle,draw,inner sep=2.4pt] (char) {#1}; \node[shape=rectangle,draw,inner sep=1pt] (char) {#1};}}

\newcommand{\C}{{\mathbb C}}       % Field of complex numbers
\newcommand{\R}{{\mathbb R}}       % Field of real numbers
\newcommand{\N}{{\mathbb N}}       % Natural numbers
\newcommand{\Z}{{\mathbb Z}}       % Ring of integer numbers
\newcommand{\DDD}{{\mathbb D}}
\newcommand{\diam}{{\rm diam}}
\newcommand{\dist}{{\rm dist}}

\newcommand{\real}{{\rm Re \,}}
\newcommand{\imag}{{\rm Im}}
\newcommand{\rf}[1]{{(\ref{#1})}}
\newcommand{\supp}{{\rm supp}}

\newcommand{\Beurling}{{\mathcal B}}
\newcommand{\Cauchy}{{\mathcal C}}
\newcommand{\norm}[1]{{\left\| {#1} \right\|}}
\newtheorem{theorem}{Theorem}%[Section]
\newtheorem*{theorem*}{Theorem}%[Section]
\newtheorem{lemma}[theorem]{Lemma}
\newtheorem{claim}[theorem]{Claim}

\newtheorem{corollary}[theorem]{Corollary}
\newtheorem*{corollary*}{Corollary}
\newtheorem{proposition}[theorem]{Proposition}
\newtheorem{definition}[theorem]{Definition}

\newtheorem{example}[theorem]{Example}
\newtheorem{remark}[theorem]{Remark}

\numberwithin{subsection}{section}
\numberwithin{theorem}{section}
\numberwithin{equation}{section}
\numberwithin{figure}{section}

\usepackage[affil-it]{authblk}

\title{Sobolev regularity of the Beurling transform on planar domains}

\author{Mart\'i Prats
\thanks{Departament de Ma\-te\-m\`a\-ti\-ques, Universitat Aut\`onoma de Bar\-ce\-lo\-na, Catalonia, currently at Departamento de Ma\-te\-m\'a\-ti\-cas, Universidad Aut\'onoma de Madrid, Spain: \texttt{marti.prats@uam.es}.}}

\begin{document}
\maketitle
\bibliographystyle{alpha}

\begin{abstract} 
Consider a Lipschitz domain $\Omega$ and the Beurling transform of its characteristic function $\Beurling \chi_\Omega(z)= - {\rm p.v.}\frac1{\pi z^2}*\chi_\Omega (z) $. It is shown that if the outward unit normal vector $N$ of the boundary of the domain is in the trace space of $W^{n,p}(\Omega)$ (i.e., the Besov space $B^{n-1/p}_{p,p}(\partial\Omega)$) then $\mathcal{B} \chi_\Omega \in W^{n,p}(\Omega)$. Moreover, when $p>2$ the boundedness of the Beurling transform on $W^{n,p}(\Omega)$ follows. This fact has far-reaching consequences in the study of the regularity of quasiconformal solutions of the Beltrami equation.


\end{abstract}

\section{Introduction}
Given a function $g\in L^p$, its Beurling transform is defined as
\begin{equation*}
\Beurling g(z):=\lim_{\varepsilon\to 0} \frac{-1}{\pi}\int_{|w-z|>\varepsilon} \frac{g(w)}{(z-w)^2}dm(w) \mbox{\quad\quad for almost every }z\in \C.
\end{equation*}
The Beurling transform is a bounded operator on $L^p$  for $1<p<\infty$ and, since it is a convolution operator, it is also bounded on the Sobolev space $W^{n,p}$ for $n\in\N$, that is, the space of functions with weak derivatives up to order $n$ in $L^p$. However, given a domain $\Omega$, the Beurling transform restricted to the domain $\Beurling_\Omega:=\chi_\Omega\Beurling(\chi_\Omega\cdot)$ is not bounded on $W^{n,p}(\Omega)$ in general, although some conditions on the regularity of the boundary of $\Omega$ can make it happen.

Consider for example the Beurling transform of the characteristic function of a square $Q$ with vertices $w_i$ for $i\in\{1,2,3,4\}$. Then, for every $z\in\Omega$ we have that $\Beurling \chi_Q (z)=\sum_i a_i \log(z-w_i)$ for some $a_i \in \C$ (see \cite[formula (4.122)]{AstalaIwaniecMartin}, for instance). Then,  $\partial \Beurling \chi_Q (z)=\sum_i a_i \frac{1}{z-w_i}$ which is not in $L^p$ for $p\geq 2$. For $n\geq 2$, the $n$-th derivative satisfies $\left|\partial^n\Beurling \chi_Q (z)\right|\approx\sum \frac{1}{|z-w_i|^n}$ which is not in $L^p$ for any $p\geq 1$. Of course, this implies that $\Beurling_Q$ is not bounded on $W^{1,p}(Q)$ for $p\geq 2$ neither on $W^{n,p}(Q)$ for $p\geq 1$ and $n\geq 2$. The interested reader may find a discussion on the case $p<2$, $n=1$ in \cite{PratsTolsa}. That paper treats also the case of the domain being the unit disk $\DDD$, when $\Beurling_\DDD$ is bounded in every Sobolev space $W^{n,p}(\DDD)$ with $1<p<\infty$. It is clear that the regularity of the boundary of a domain $\Omega$ plays a crucial role in determining whether the restricted Beurling transform is bounded or not on $W^{n,p}(\Omega)$.

In \cite{CruzMateuOrobitg}  Cruz, Mateu and Orobitg proved a $T(1)$-theorem for domains with parameterizations of the boundary of $\Omega$ in $C^{1,\varepsilon}$ with $0<\varepsilon<1$ that grants the boundedness of $\Beurling_\Omega$ in the Sobolev space $W^{s,p}(\Omega)$  if $\Beurling_\Omega1=\chi_\Omega\Beurling\chi_\Omega \in W^{s,p}(\Omega)$ for $0<s\leq 1$ and $1<p<\infty$ with $s p>2$ (the Sobolev space is defined via the Bessel potential for $s\notin \N$).
Moreover, they showed that when $0<s<\varepsilon<1$ and $1<p<\infty$ one has that $\Beurling_\Omega1 \in W^{s,p}(\Omega)$ by means of some results from \cite{MateuOrobitgVerdera}.

In \cite{PratsTolsa}, Tolsa and the author of the present text presented a $T(P)$-theorem for $W^{n,p}(\Omega)$ which is valid for Lipschitz domains (and uniform domains as well) when $n\in\N$ and $p>2$, granting the boundedness of $\Beurling_\Omega$ on $W^{n,p}(\Omega)$ if $\Beurling_\Omega P \in W^{n,p}(\Omega)$ for every polynomial $P$ of degree smaller than $n$.

Cruz and Tolsa proved in \cite{CruzTolsa} that for $0<s\leq1$, $1<p<\infty$ with $sp>1$, if the outward unit normal vector $N$ is in the Besov space ${B}^{s-1/p}_{p,p}(\partial\Omega)$ (see Section \ref{secSpaces}) then $\Beurling\chi_\Omega\in W^{s,p}(\Omega)$.
This condition is necessary for Lipschitz domains with small Lipschitz constant (see \cite{TolsaSharp}). 
Moreover, being $N\in B^{s-1/p}_{p,p}(\partial\Omega)$ implies the parameterizations of the boundary of $\Omega$ to be in $B^{s+1-1/p}_{p,p}$ and, for  $sp>2$, the parameterizations are in $C^{1,s-2/p}$ by the Sobolev Embeding Theorem. In that situation, one can use the $T(1)$ result in \cite{CruzMateuOrobitg} to deduce the boundedness of the Beurling transform in $W^{s,p}(\Omega)$.

In this article we prove that the result in \cite{CruzTolsa} holds for $s\in \N$:
\begin{theorem}\label{theoGeometricVeryNaive}
Let $p>1$, let $n\in \N$ and let $\Omega$ be a bounded  Lipschitz domain with parameterizations in $C^{n-1,1}$ and with  $N\in B^{n-1/p}_{p,p}(\partial\Omega)$. Then,  we have that
\begin{equation*}
\norm{\Beurling(\chi_\Omega)}_{W^{n,p}(\Omega)}\leq C \norm{N}_{B^{n-1/p}_{p,p}(\partial\Omega)},
\end{equation*}
where $C$ depends on $p$, $n$, $\diam(\Omega)$ and the Lipschitz character of the domain.
\end{theorem}

 The proof presented here will be slightly more tricky since we will need to approximate the boundary of the domain by polynomials instead of straight lines. The derivative of the Beurling transform of the characteristic function of a half-plane is zero out of its boundary (see \cite{CruzTolsa}), but the derivative of the Beurling transform of the characteristic function of a domain bounded by a polynomial of degree greater than one is not zero anymore in general. 

Using the $T(P)$-theorem of \cite{PratsTolsa} this will suffice to see the boundedness of the Beurling transform.
\begin{theorem}\label{theoGeometricNaive}
Let  $2<p<\infty$, let $n\in \N$ and let $\Omega$ be a bounded Lipschitz domain with  $N\in B^{n-1/p}_{p,p}(\partial\Omega)$. Then, for every $f\in W^{n,p}(\Omega)$ we have that
\begin{equation*}
\norm{\Beurling(\chi_\Omega f)}_{W^{n,p}(\Omega)}\leq C \norm{N}_{B^{n-1/p}_{p,p}(\partial\Omega)}\norm{f}_{W^{n,p}(\Omega)},
\end{equation*}
where $C$ depends on $p$, $n$, $\diam(\Omega)$ and the Lipschitz character of the domain.
\end{theorem}

Both theorems above are particular cases of Theorems  \ref{theoGeometric} and \ref{theoGeometricPGtr2}, which cover a wider family of operators including the Beurling transform and its iterates $\Beurling^m$, showing that the constants have exponential growth with respect to $m$ with base as close to $1$ as desired. This has far-reaching consequences in quasiconformal mappings. 

Indeed, let $\mu\in L^\infty$ supported in a certain ball $B\subset \C$ with $\norm{\mu}_{L^\infty}<1$. We say that $f$ is a quasiregular solution to the Beltrami equation
\begin{equation}\label{eqBeltrami}
\bar{\partial} f =\mu\, \partial f
\end{equation}
with Beltrami coefficient $\mu$ if $f\in W^{1,2}_{loc}$, that is, if $f$ and $\nabla f$ are square integrable functions in any compact subset of $\C$, and $\bar{\partial} f (z)=\mu(z) \partial f(z)$ for almost every $z\in\C$. Such a function $f$ is said to be a quasiconformal mapping if it is a homeomorphism of the complex plane. If, moreover, $f(z)=z+\mathcal{O}(\frac1z)$ as $z\to\infty$, then we say that $f$ is the principal solution to \rf{eqBeltrami}. 

Given a compactly supported Beltrami coefficient $\mu$, the existence and uniqueness of the principal solution is granted by the measurable Riemann mapping Theorem (see \cite[Theorem 5.1.2]{AstalaIwaniecMartin}, for instance). The operator $I-\mu \Beurling$ is invertible in $L^2$ and, if we call 
\begin{equation*}
h:=(I-\mu \Beurling)^{-1} \mu=\mu+\mu\Beurling(\mu)+\mu\Beurling (\mu\Beurling(\mu))+\cdots,
\end{equation*}
and $f$ is the principal solution of \rf{eqBeltrami} then $\bar\partial f= h$ and  $\partial f= \Beurling h +1$.

Let $n, m\in\N$ and $2<p<\infty$. In \cite{PratsQuasiconformal}, the author of the present article uses the results obtained here to show that if a domain  $\Omega$ satisfies the hypothesis of Theorem \ref{theoGeometricNaive} and a Beltrami coefficient $\mu\in W^{n,p}(\Omega)$, then  $\mu^m \Beurling^m$ is a bounded operator on $W^{n,p}(\Omega)$ with norm tending to zero as $m$ tends to infinity. This is used to show that $h\in W^{n,p}(\Omega)$ as well by means of Fredholm theory, giving place to the following remarkable result.
\begin{theorem*}[See \cite{PratsQuasiconformal}]
Let $n\in \N$, let $\Omega$ be a bounded Lipschitz domain with outward unit normal vector $N$ in $B^{n-1/p}_{p,p}(\partial\Omega)$ for some $2<p<\infty$ and let $\mu\in W^{n,p}(\Omega)$ with $\norm{\mu}_{L^\infty}<1$ and $\supp(\mu)\subset\overline{\Omega}$. Then, the operator 
$$(I_\Omega-\mu \Beurling_\Omega):f \mapsto \left(\chi_\Omega f -\mu \Beurling_\Omega( f)\right)$$
 is invertible in $W^{n,p}(\Omega)$ and the principal solution  $f$ to \rf{eqBeltrami} is in the Sobolev space $W^{n+1,p}(\Omega)$.
\end{theorem*}

For results connecting the Sobolev regularity $W^{s,p}(\C)$ of a quasiconformal mapping and its Beltrami coefficient we refer the reader to \cite{Astala}, \cite{AstalaIwaniecSaksman}, \cite{ClopFaracoMateuOrobitgZhong}, \cite{ClopFaracoRuiz} and \cite{CruzMateuOrobitg} and, when Sobolev spaces on domains are concerned, to \cite{MateuOrobitgVerdera}, \cite{CittiFerrari} and \cite{CruzMateuOrobitg} again.

The plan of the paper is the following. In Section \ref{secPreliminaries} some preliminary assumptions are stated. Subsection \ref{secNotation} explains the notation to be used and recalls some well-known facts. In Subsection \ref{secBetas} one finds the definition of some generalized $\beta$-coefficients related to Jones and David-Semmes' celebrated betas. In Subsection \ref{secSpaces} the definition of the Besov spaces $B^s_{p,p}$ is given along with some related well-known facts and an equivalent norm in terms of the generalized $\beta$-coefficients using a result by Dorronsoro in \cite{Dorronsoro}. Subsection \ref{secOperators} is about some operators related to the Beurling transform, providing a standard notation for the whole article.

Section \ref{secCharacteristic} is devoted to prove Theorems \ref{theoGeometricVeryNaive} and  \ref{theoGeometricNaive}. The first step is to study the case of unbounded domains whose boundary can be expressed as the graph of a Lipschitz function. Subsection \ref{secCharacteristicSpecial} contains the outline of the proof, reducing it to two lemmas. The first one studies the relation with the $\beta$-coefficients and is proven in Subsection \ref{secInterstitial}. The second one, proven in Subsection \ref{secPoly}, is about the case where the domain is bounded by the graph of a polynomial, and here one finds  the exponential behavior of the bounds for the iterates of the Beurling transform, which entangles the more subtle details of the proof. Finally, in Subsections \ref{secLocalization} and \ref{secPGtr2} one finds a more quantitative version of Theorem \ref{theoGeometricVeryNaive} and Theorem \ref{theoGeometricNaive} for bounded Lipschitz domains using a localization principle and the aforementioned $T(P)$-theorem.

\section{Preliminaries}\label{secPreliminaries}
\subsection{Some notation and well-known facts}\label{secNotation}
{\bf On inequalities:} 
When comparing two quantities $x_1$ and $x_2$ that depend on some parameters $p_1,\dots, p_j$ we will write 
$$x_1\leq C_{p_{i_1},\dots, p_{i_j}} x_2$$
if the constant $C_{p_{i_1},\dots, p_{i_j}} $ depends on ${p_{i_1},\dots, p_{i_j}}$. We will also write $x_1\lesssim_{p_{i_1},\dots, p_{i_j}} x_2$ for short, or simply $x_1\lesssim  x_2$ if the dependence is clear from the context or if the constants are universal. We may omit some of these variables for the sake of simplicity. The notation $x_1 \approx_{p_{i_1},\dots, p_{i_j}} x_2$ will mean that $x_1 \lesssim_{p_{i_1},\dots, p_{i_j}} x_2$ and $x_2 \lesssim_{p_{i_1},\dots, p_{i_j}} x_1$.

{\bf On polynomials:}
We write $\mathcal{P}^n$ for the vector space of polynomials of degree smaller or equal than $n$ with one variable. 
%We write $\mathcal{P}^n(\R^d)$ for the vector space of real polynomials of degree smaller or equal than $n$ with $d$ real variables. If it is clear from the context we will just write $\mathcal{P}^n$. 

{\bf On sets:}
Given two sets $A$ and $B$, their {\rm symmetric difference} is $A\Delta B:=(A\cup B) \setminus (A \cap B)$. Given $z\in \C$ and $r>0$, we write $B(z,r)$ or $B_r(z)$ for the open ball centered at $z$ with radius $r$ and $Q(z,r)$ for the open cube centered at $z$ with sides parallel to the axis and side-length $2r$. Given any cube $Q$, we write $\ell(Q)$ for its side-length, and $rQ$ will stand for the cube with the same center but enlarged by a factor $r$. We will use the same notation for balls and one dimensional cubes, that is, intervals. For instance, $I(x,r)=(x-r,x+r)$ for $x\in\R$ and $r>0$. 

At some point we need to use intervals in $\C$: given $z,w\in \C$, we call the interval with endpoints $z$ and $w$
$$[z,w]:=\{(1-t)z + tw: t\in[0,1]\}.$$
We may use the ``open'' interval $]z,w[:=[z,w]\setminus\{z,w\}$.

Let $n\in\N$. We say that  a function $f:\R\to \C$ belongs to the Lipschitz class $C^{n-1,1}$ if it has $n-1$ continuous derivatives and
$$\norm{f}_{C^{n-1,1}(\R)}=\sum_{i=1}^{n-1}\norm{f^{(i)}}_{L^\infty(\R)}+ \sup_{\substack{z,w\in\R \\z\neq w}}\frac{|f^{(n-1)}(z)-f^{(n-1)}(w)|}{|z-w|}.$$

 We call domain an open and connected subset of $\C$.
\begin{definition}\label{defLipschitz}
Given $n\geq 1$, we say that $\Omega\subset\C$ is a $(\delta,R)-C^{n-1,1}$ domain if given any $z\in\partial\Omega$, there exists a function $A_z\in C^{n-1,1}(\R)$ supported in $[-4R,4R]$ such that 
\begin{equation*}
\norm{A_z^{(j)}}_{L^\infty}\leq \frac{\delta}{R^{j-1}} \mbox{\,\,\,\, for every } 0\leq j \leq n,
\end{equation*}
and, possibly after a rigid movement $\tau$ composed by a translation that sends $z$ to the origin and a rotation that brings the tangent at $z$ to the real line, we have that
$$\tau(\Omega)\cap Q(0,R)= \{x+i\,y: y>A_z(x)\},$$
and so that, given $|x|\leq R$, the point in the graph $(x,A(x))$ belongs to $\partial\Omega$ after the corresponding rotation and translation.
In case $n=1$ the assumption of the tangent is removed (we say that $\Omega$ is a $(\delta,R)$-Lipschitz domain). 

We call {\em window} the preimage $\mathcal{Q}=\tau^{-1}(Q(0,R))$ by that rigid movement.
\end{definition}

{\bf On measure theory:}
We denote the $1$-dimensional Lebesgue measure in $\R$ by $m_1$ (or $m$ if it is clear from the context).  We will write $dz$ for the form $dx+i\,dy$ and analogously $d\bar{z}=dx-i\,dy$, where $z=x+i\,y$. Thus, when integrating a function with respect to  the Lebesgue measure of a complex variable $z$ we will always use $dm(z)$ to avoid confusion, or simply $dm$. Note that, at some point, we use $m$ also to denote a natural number.

{\bf On indices:}
In this text  $\N_0$ stands for the natural numbers including $0$. Otherwise we will write $\N$.  We will make wide use of the multiindex notation for exponents and derivatives. For $\alpha\in\Z^2$  its modulus is $|\alpha|=\sum_{i=1}^2|\alpha_i|$ and its factorial is $\alpha!=\alpha_1!\alpha_2!$. Given two multiindices $\alpha, \gamma\in \Z^2$ we write $\alpha\leq \gamma$ if $\alpha_i\leq \gamma_i$ for every $i$. We say $\alpha<\gamma$ if, in addition, $\alpha\neq\gamma$. Furthermore, we write
$${\alpha \choose \gamma}:=\prod_{i=1}^2 {\alpha_i \choose \gamma_i}=\begin{cases}
\prod_{i=1}^2 \frac{\alpha_i!}{\gamma_i!(\alpha_i-\gamma_i)!}  & \mbox{if }\alpha \in \N_0^2 \mbox{ and }\vec{0}\leq \gamma \leq\alpha ,\\
0 & \mbox{otherwise.}
\end{cases}$$
%For $z\in \C$ and $\alpha\in \Z^2$ we write $x^\alpha:= \prod x_i^{\alpha_i}$. 
% Given any $\phi\in C^\infty_c$  and $\alpha\in\N_0^d$ we write $D^\alpha\phi=\frac{\partial^{|\alpha|}}{\prod \partial_{x_i}^{\alpha_i}}\phi$. 

At some point we will use also roman letter for multiindices, and then, to avoid confusion, we will use the vector notation $\vec{i},\vec{j}, \dots$

{\bf On complex notation}
For $z=x+i\,y \in \C$ we write $\real(z):=x$ and $\imag(z):=y$. Note that the symbol $i$ will be used also widely as a index for summations without risk of confusion. The multiindex notation will change slightly: for $z\in \C$ and $\alpha\in \Z^2$ we write $z^\alpha:=z^{\alpha_1}\bar{z}^{\alpha_2}$. 

We also adopt the traditional Wirtinger notation for derivatives, that is, given any $\phi\in C^\infty_c(\C)$, then 
$$\partial \phi (z):=\frac{\partial \phi}{\partial z}(z)=\frac12(\partial_x\phi-i\,\partial_y\phi) (z),$$
 and 
 $$\bar \partial \phi (z):=\frac{\partial\phi}{\partial \bar z}(z)=\frac12(\partial_x\phi+i\,\partial_y\phi) (z).$$
Thus, given any $\phi\in C^\infty_c(\C)$  (infintitely many times differentiable with compact support in $\C$) and $\alpha\in\N_0^2$, we write $D^\alpha\phi=\partial^{\alpha_1}\bar\partial^{\alpha_2}\phi$.

{\bf On Sobolev spaces:}
For any open set $U\subset \C$, every distribution $f\in \mathcal{D}'(U)$ and $\alpha\in\N_0^2$, the {\em distributional derivative} $D^\alpha_U f$ is the distribution defined by
$$\langle D^\alpha_U f,\phi\rangle:=(-1)^{|\alpha|}\langle f, D^\alpha \phi\rangle \mbox{\,\,\,\, for every }\phi \in C^\infty_c(U).$$
Abusing notation we will write $D^\alpha$ instead of $D^\alpha_U$ if it is clear from the context. If the distribution is regular, that is, if it coincides with an $L^1_{loc}$ function acting on $\mathcal{D}(U)$, then we say that $D^\alpha_U f$ is a {\em weak derivative} of $f$ in $U$. We write $|\nabla^n f|=\sum_{|\alpha|=n}|D^\alpha f|$.

Given numbers $n\in\N$, $1\leq p\leq\infty$ an open set $U\subset\C$ and an $L^1_{loc}(U)$ function $f$, we say that $f$ is in the Sobolev space $W^{n,p}(U)$ of smoothness $n$ and order of integrability $p$ if $f$ has weak derivatives $D^\alpha_U f\in L^p$ for every $\alpha\in\N_0^2$ with $|\alpha|\leq n$. When $\Omega$ is a Lipschitz domain, we will use the norm
$$\norm{f}_{W^{n,p}(\Omega)}=\norm{f}_{L^p(\Omega)}+\norm{\nabla^n f}_{L^p(\Omega)},$$
which is equivalent to considering also the fewer order derivatives, that is, 
\begin{equation}\label{eqEquivalenceNormsSobolev}
\norm{f}_{W^{n,p}(\Omega)}\approx \norm{f}_{L^p(\Omega)}+\sum_{|\alpha|\leq n}\norm{D^\alpha f}_{L^p(\Omega)}
\end{equation}
(see \cite[Theorem 4.2.4]{TriebelInterpolation}) or, if $\Omega$ is an extension domain,
$$\norm{f}_{W^{n,p}(\Omega)}\approx \inf_{F: F|_\Omega\equiv f}\norm{F}_{W^{n,p}(\C)}.$$
From \cite{Jones}, we know that uniform domains (and in particular, Lipschitz domains) are Sobolev extension domains for any indices $n\in \N$ and $1\leq p \leq\infty$. One can find deeper results in that sense in \cite{Shvartsman} and \cite{KoskelaRajalaZhang}.

The reader can consider $n\in\N$ and $1<p<\infty$ to be two given numbers along the whole text. At some point the restriction $2<p$ will be needed.

{\bf On finite diferences:}
Given  a function $f:\Omega\subset \C \to \C$ and two values $z, h\in \C$ such that $[z, z+h]\subset \Omega$, we call 
$$\Delta_h^1 f (z)=\Delta_h f (z)=f(z+h)-f(z).$$
Moreover, for any natural number $i\geq2$ we define the iterated difference
$$\Delta_h^i f (z)=\Delta_h^{i-1}f(z+h)-\Delta_h^{i-1}f(z)=\sum_{j=0}^i (-1)^{i-j} {i \choose j} f(z+j h)$$
whenever the segment $[z,z+ih]\subset \Omega$.
 
{\bf On Whitney coverings:}
Given a domain $\Omega$, we say that a collection of open dyadic cubes $\mathcal{W}$ is a {\rm Whitney covering} of $\Omega$ if they are disjoint, the union of the cubes and their boundaries is $\Omega$, there exists a constant $C_{\mathcal{W}}$ such that 
$$C_\mathcal{W} \ell(Q)\leq \dist(Q, \partial\Omega)\leq 4C_\mathcal{W}\ell(Q),$$
two neighbor cubes $Q$ and $R$ (i.e., $\overline Q\cap \overline R\neq\emptyset$) satisfy $\ell(Q)\leq 2 \ell(R)$, and the family $\{20 Q\}_{Q\in\mathcal{W}}$ has finite superposition. The existence of such a covering is granted for any open set different from $\C$ and in particular for any domain as long as $C_\mathcal{W}$ is big enough (see \cite[Chapter 1]{SteinPetit} for instance).

{\bf On the Leibniz rule:}
The Leibniz formula (see \cite[Section 5.2.3]{Evans}) says that given a domain $\Omega\subset \C$, a function $f\in W^{n,p}(\Omega)$, a multiindex $\alpha\in \N_0^2$ with $|\alpha| \leq n$ and $\phi\in C^\infty_c(\Omega)$, we have that $\phi\cdot f\in  W^{n,p}(\Omega)$ with 
\begin{equation*}
D^\alpha (\phi\cdot f)=\sum_{\gamma\leq \alpha}{\alpha\choose\gamma} D^\gamma\phi D^{\alpha-\gamma} f.
\end{equation*}

{\bf On Green's formula:}
Green's Theorem can be written in terms of complex derivatives (see \cite[Theorem 2.9.1]{AstalaIwaniecMartin}). Let $\Omega$ be a bounded Lipschitz domain. If $f, g\in W^{1,1}(\Omega)\cap C(\overline{\Omega})$, then 
\begin{equation}\label{eqGreen}
\int_\Omega \left(\partial f + \bar \partial g\right) \,dm=\frac{i}2\left(\int_{\partial\Omega} f(z) \, d\bar{z}- \int_{\partial\Omega}g(z)\, dz\right).
\end{equation}

{\bf On Rolle's Theorem:}
We state here also a Complex Rolle Theorem for holomorphic functions \cite[Theorem 2.1]{EvardJafari} that will be a cornerstone of Section \ref{secPoly}.
\begin{theorem}\label{theoRolle}[see \cite{EvardJafari}]
Let $f$ be a holomorphic function defined on an open convex set $U\subset\C$. Let $a,b\in U$ such that $f(a)=f(b)=0$ and $a\neq b$. Then there exists $z$ in the segment  $]a,b[$ such that $\real(\partial f(z))=0$. 
\end{theorem}

{\bf On the Sobolev Embedding Theorem:}
We state a reduced version of the Sobolev Embedding Theorem for Lipschitz domains (see \cite[Theorem 4.12, Part II]{AdamsFournier}). 
For each Lipschitz domain $\Omega\subset \C$ and every $p>2$, there is a continuous embedding of the Sobolev space $W^{1,p}(\Omega)$ into the H\"older space $C^{0,1-\frac{2}{p}}(\overline\Omega)$. That is, writing
$$\norm{f}_{C^{0,s}(\overline\Omega)}:=\norm{f}_{L^\infty(\Omega)}+\sup_{\substack{z,w\in\overline\Omega\\z\neq w}}\frac{|f(z)-f(w)|}{|z-w|^s}\mbox{\,\,\,\, for $0<s\leq 1$},$$
we have that for every $f\in W^{1,p}(\Omega)$, 
\begin{equation*}
\norm{f}_{C^{0,1-\frac 2p}(\overline\Omega)}\leq C_\Omega \norm{f}_{W^{1,p}(\Omega)}.
\end{equation*}

{\bf On inequalities:}
We will use Young's Inequality. It states that for measurable functions $f$ and $g$, we have that
\begin{equation}\label{eqYoung}
\norm{f * g}_{L^q}\leq\norm{f}_{L^r}\norm{g}_{L^p}
\end{equation}
for $1\leq p,q,r\leq \infty$ with $\frac{1}{q}=\frac{1}{p}+\frac{1}{r}-1$ (see \cite[Appendix A2]{SteinPetit}).

\subsection{Some generalized betas}\label{secBetas}
In \cite{Dorronsoro}, Dorronsoro introduces a characterization of Besov spaces in terms of the mean oscillation of the functions on cubes, and he uses approximating polynomials to do so. If the polynomials are of degree one, that is straight lines, this definition can be written in terms of a certain sum of David-Semmes betas (see \cite{CruzTolsa} for instance). Following the ideas of Dorronsoro in our case we will use higher degree polynomials to approximate the Besov function that we want to consider, giving rise to some generalized betas. The following proposition comes from \cite{Dorronsoro}, where it is not explicitly proven. We give a short proof of it for the sake of completeness.

\begin{proposition}\label{propoOrthoApproxPoly}
Given a locally integrable function $f: \R \to \R$ and an interval $I\subset \R$, there exists a unique polynomial $\mathbf{R}^n_I f \in \mathcal{P}^n$ which we will call \emph{approximating polynomial} of $f$ on $I$, such that given any $j\leq n$ one has that
\begin{equation}\label{eqOrthoApproxPoly}
\int_I (\mathbf{R}^n_I f- f) \, x^j=0.
\end{equation}
\end{proposition}

\begin{remark}\label{remApproxPoly}
In case of existence, the approximating polynomial verifies
$$\sup_{x\in I} |\mathbf{R}^n_I f(x)|\leq C_{n} \frac{1}{|I|} \int_I |f| \, dm.$$
\end{remark}
\begin{proof}
Indeed, since $\mathcal{P}^n$ is a finite dimensional vectorial space, all the norms are equivalent. In particular one can easily see that for any $P\in \mathcal{P}^n$
$$\norm{P}_{L^\infty(I)}^2\approx \frac{1}{|I|}\norm{P}^2_{L^2(I)}.$$
Using the linearity of the integral in \rf{eqOrthoApproxPoly}, one has
$$\frac{1}{|I|}\int_I|\mathbf{R}^n_I f|^2 \, dm=\frac{1}{|I|}\int_I \mathbf{R}^n_I f \cdot f \, dm.$$
Combining both facts one gets
$$\norm{\mathbf{R}^n_I f}_{L^\infty(I)}^2\lesssim \frac1{|I|}\norm{\mathbf{R}^n_I f}_{L^\infty(I)}\norm{f}_{L^1(I)}.$$
\end{proof}

\begin{proof}[Proof of Proposition \ref{propoOrthoApproxPoly}]
By the Hilbert Projection Theorem, $L^2(I)=\mathcal{P}^n \oplus (\mathcal{P}^n)^{\perp}$. Thus, if $f\in L^2(I)$, we can write
$f|_I=\mathbf{R}^n_I f + (f|_I-\mathbf{R}^n_I f)$ satisfying \rf{eqOrthoApproxPoly}. 

For general $f\in L^1$, we can define a sequence of functions $\{f_j\}_{j\in \N}\subset L^2(I)$ such that $|f_j|\leq|f|$ and  $f_j\xrightarrow{a.e.} f$.
By Remark \ref{remApproxPoly} we have that the approximating polynomials $\mathbf{R}^n_I f_j$ are uniformly bounded in $I$ by 
$$\sup_{x\in I} |\mathbf{R}^n_I f_j(x)|\lesssim\frac{1}{|I|} \int_I |f_j| \, dm\leq\frac{1}{|I|} \int_I |f| \, dm.$$
Therefore there exists a convergent subsequence of $\{\mathbf{R}^n_I f_j\}_j$ in $L^1$ (and in any other norm). We call $\mathbf{R}^n_I f$ the limit of one such partial. By the Dominated Convergence Theorem we get \rf{eqOrthoApproxPoly}.

To see uniqueness, we observe that if we find two polynomials $P_1$ and $P_2$ satisfying \rf{eqOrthoApproxPoly}, then
$$\int_I (P_1-P_2)P\, dm=0$$
for any $P\in\mathcal{P}^n$. In particular, if we take $P=P_1-P_2$ we get that $\norm{P_1-P_2}_{L^2(I)}=0$.
\end{proof}

\begin{remark}
Given $P\in \mathcal{P}^n$, an interval $I$ and  $1\leq p\leq\infty$ we have that
\begin{equation}\label{eqApproxPolyOptimal}
\norm{f-\mathbf{R}^n_I f}_{L^p(I)} \leq C_{n} \norm{f-P}_{L^p(I)},
\end{equation}
and given any intervals $I\subset I'$, 
\begin{equation}\label{eqApproxPolyChain}
\norm{f-\mathbf{R}^n_I f}_{L^p(I)} \leq C_{n} \norm{f-\mathbf{R}^n_{I'} f}_{L^p(I')}.
\end{equation}
\end{remark}

\begin{proof}
By means of the Triangle Inequality and \rf{eqOrthoApproxPoly}, we have that for any $P\in\mathcal{P}^n$
\begin{align*}
\norm{f-\mathbf{R}^n_I f}_{L^p(I)} 
	& \leq \norm{f-P}_{L^p(I)}  +\norm{P-\mathbf{R}^n_I f}_{L^p(I)} =\norm{f-P}_{L^p(I)}+ \norm{\mathbf{R}^n_I(P- f)}_{L^p(I)}.
\end{align*}
Therefore, we use twice H\"older's Inequality and Remark \ref{remApproxPoly} to get
\begin{align*}
\norm{f-\mathbf{R}^n_I f}_{L^p(I)} 
	& \leq \norm{f-P}_{L^p(I)}+ |I|^{1/p}\norm{\mathbf{R}^n_I(P- f)}_{L^\infty(I)}\\
	& \lesssim_{n} \norm{f-P}_{L^p(I)}+ \frac{|I|^{1/p}}{|I|}\norm{P- f}_{L^1(I)} \leq 2 \norm{f-P}_{L^p(I)}.
\end{align*}

The  inequality \rf{eqApproxPolyChain} is just a consequence of \rf{eqApproxPolyOptimal} replacing $P$ by $\mathbf{R}^n_{I'} f$.
\end{proof}

\begin{remark}\label{remCut}
This proposition is still valid in any dimension mutatis mutandis. However, in the one dimensional case, if $f$ is continuous and $I$ is an interval one can easily see that $f-\mathbf{R}^n_I f$ has $n+1$ zeroes at least. Indeed, if it did not happen, one could find a polynomial $P\in \mathcal{P}^n$ with a simple zero at every point where $f-\mathbf{R}^n_I f$ changes its sign, and no more. Therefore, $(f-\mathbf{R}^n_I f)\cdot P$ would have constant sign and, thus,  the integral in \rf{eqOrthoApproxPoly} would not vanish (see Figure \ref{figpolinomidorronsoro}). 
\end{remark}
\begin{figure}[ht]
 \centering
  \begin{subfigure}[b]{0.4\textwidth}{\includegraphics[width=\textwidth]{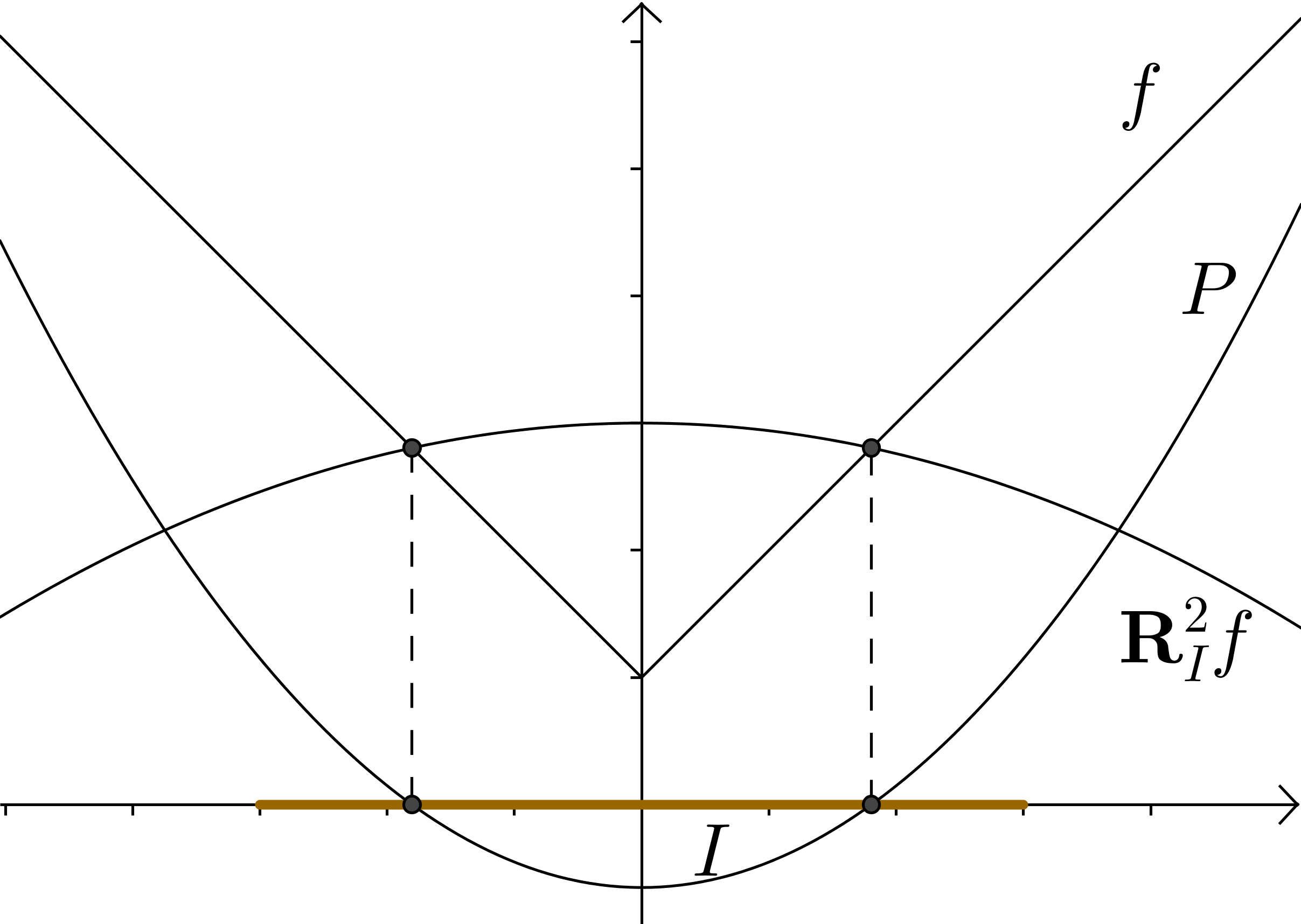}}\end{subfigure}
  \quad\quad\quad\quad \begin{subfigure}[b]{0.4\textwidth}{\includegraphics[width=\textwidth]{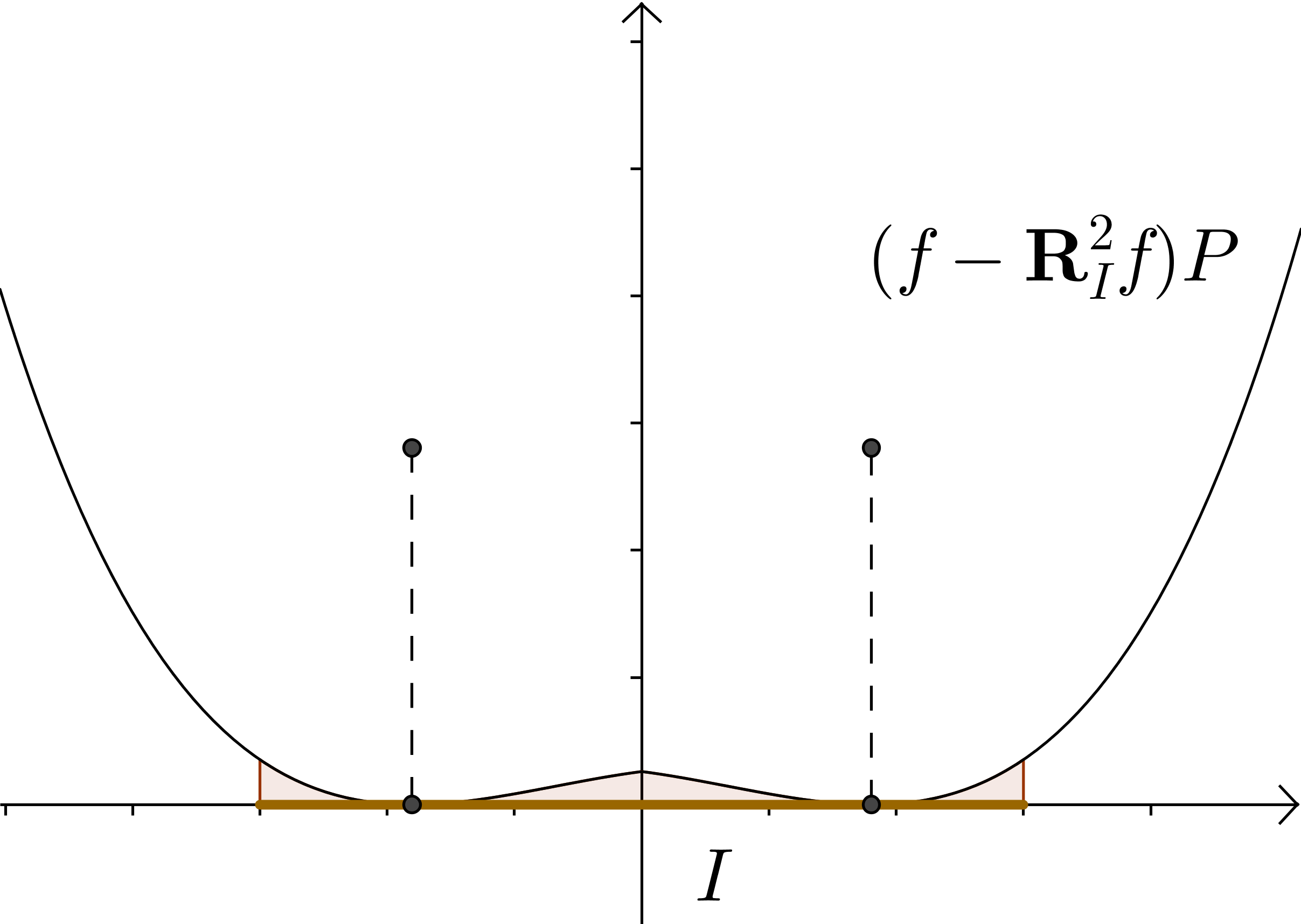}}\end{subfigure}    
  \caption{If $f-\mathbf{R}^2_I f$ had only $2$ zeroes, there would exist $P\in \mathcal{P}^2$ with $\int_I (f-\mathbf{R}^2_I f)P \, dm>0$.}\label{figpolinomidorronsoro}
\end{figure}

Now we can define the generalized betas.

\begin{definition}
Let $f: \R\to \R$ be a locally integrable function and $I\subset \R$ an interval. Then we define
$$\beta_{(n)}(f,I):=\frac{1}{|I|}\int_{3I}\frac{|f(x)-\mathbf{R}^n_{3I} f(x)|}{|I|} \, dm(x).$$
\end{definition}

\begin{remark}\label{remBetas}
Taking into account \rf{eqApproxPolyOptimal}, we can conclude that
$$\beta_{(n)}(f,I)\approx \inf_{P\in\mathcal{P}^n}\frac{1}{|I|}\int_{3I} \frac{|f(x) - P(x)|}{|I|} dm(x) .$$
This can be seen as a generalization of David and Semmes $\beta_1$ coefficient since $\beta_{(1)}$ and $\beta_{1}$ are comparable as long as some Lipschitz condition is assumed on $f$.
\end{remark}

\subsection{Function spaces}\label{secSpaces}
Next we recall some definitions and results on the function spaces that we will use. For a complete treatment we refer the reader to \cite{TriebelTheory} and \cite{RunstSickel}.

\begin{definition}\label{defCollection}Let $\Phi(\R)$ be the collection of all the families $\Psi=\{\psi_j\}_{j=0}^\infty\subset C^\infty_c(\R)$ such that
\begin{equation*}
\left\{ 
\begin{array}{ll}
\supp \,\psi_0 \subset (-2,2), & \\
\supp \,\psi_j \subset (-2^{j+1},2^{j+1})\setminus (-2^{j-1},2^{j-1}) & \mbox{ if $j\geq 1$},\\	
\end{array}
\right.
\end{equation*}
for all  $i\in\N_0$ there exists a constant $c_i$ such that
\begin{equation*}
\norm{ \psi_j^{(i)}}_\infty \leq \frac{c_i}{2^{j i} } \mbox{\,\,\, for every $j\geq 0$,}
\end{equation*}
and
\begin{equation*}
\sum_{j=0}^\infty \psi_j(x)=1 \mbox{\,\,\, for every $x\in\R$.}
\end{equation*}
\end{definition}
 
 \begin{definition}
Given any Schwartz function $\psi \in \mathcal{S}(\R)$ its Fourier transform is
$$F\psi(\zeta)=\int_{\R} e^{-2\pi i x  \zeta} \psi(x) dm(x).$$
This notion extends to the tempered distributions $\mathcal{S}(\R)'$ by duality.
 
Let $s \in \R$,  $1\leq p\leq \infty$, $1\leq q\leq\infty$ and $\Psi \in \Phi(\R)$. For any tempered distribution $f\in \mathcal{S}'(\R)$ we define the non-homogeneous Besov space
\begin{equation*}
\norm{f}_{B^s_{p,q}}^\Psi=\norm{\left\{2^{sj}F^{-1}\psi_j F f\right\}}_{l^q(L^p)}=\norm{\left\{2^{sj}\norm{F^{-1}\psi_j F f}_{L^p}\right\}}_{l^q},
\end{equation*}
and we call $B^s_{p,q}\subset \mathcal{S}'$ to the set of tempered distributions such that this norm is finite.
\end{definition}
These norms are equivalent for diferent choices of $\Psi$. In general one works with radial $\psi_j$ and such that $\psi_{j+1}(x)=\psi_j(x/2)$. Of course we will ommit $\Psi$ in our notation since it plays no role.

\begin{proposition}[{See \cite[Sections 2.3.3 and 2.7.1]{TriebelTheory}}]
The following properties hold:
\begin{enumerate}
\item Let $1\leq q_0, q_1\leq \infty$ and $1\leq p\leq  \infty$, $s\in\R$ and $\varepsilon>0$. Then 
$$B^{s+\varepsilon}_{p,q_0}\subset B^{s}_{p,q_1}.$$
\item Given $1\leq p_0\leq p_1\leq \infty$ and $-\infty < s_1\leq s_0<\infty$. Then
\begin{equation}\label{eqDifferentialDimension}
B^{s_0}_{p_0,p_0}\subset B^{s_1}_{p_1,p_1}\mbox{\,\,\,\, if }s_0-\frac{1}{p_0}=s_1-\frac{1}{p_1}.
\end{equation}
%\item Given $\varepsilon>0$, $1< p_0< p_1\leq \infty$ and $-\infty < s_1\leq s_0<\infty$ with $s_0\in\N$ and $s_0-\frac{d}{p_0}=s_1-\frac{d}{p_1}$, then
%$$W^{s_0,p_0}\subset B^{s_0-\varepsilon}_{p_0,p_0} \mbox{\,\,\,\, and \,\,\,\,} W^{s_0,p_0}\subset B^{s_1}_{p_1,p_1}.$$
\end{enumerate}
\end{proposition}

If we set $j\in \Z$ instead of $j\in \N$ in Definition \ref{defCollection}, then we get the homogeneous spaces of tempered distributions (modulo polynomials) $\dot B^s_{p,q}$. In particular, by \cite[Theorem 2.3.3]{TriebelTheoryII} we have that  if $s>0$ then
%\cite[Theorem 2.5.12 and Theorem 5.2.3/2]{TriebelTheory}
\begin{equation}\label{eqHomogeneous}
\norm{f}_{B^s_{p,q}}\approx \norm{f}_{\dot B^s_{p,q}}+\norm{f}_{L^p} \mbox{ \,\,\,\, for any }f\in \mathcal{S}'.
\end{equation} 

In the particular case of homogeneous Besov spaces with $1\leq p,q\leq \infty$ and $s>0$, one can give an equivalent definition in terms of differences of order $M\geq \left[s\right]+1$:
\begin{equation}\label{eqBesovDiff}
\norm{f}_{\dot B^s_{p,q}} \approx \left( \int_{\R} \frac{\norm{\Delta^M_h f}^q_{L^p}}{|h|^{sq}} \frac{dm(h)}{|h|}\right)^{\frac{1}{q}}.
\end{equation}

In \cite{CruzTolsa} the authors point out that the seminorm of the homogeneous Besov space $\dot B^s_{p,q}$ for $0<s<1$ can be defined in terms of the approximating polynomials of degree $1$ from the previous section. In general, \cite[Theorem 1]{Dorronsoro} together with \rf{eqApproxPolyChain} and Remark \ref{remBetas} can be used to prove without much effort that for any $s>0$ and $n\geq \left[s\right]$, 
$$\norm{f}_{\dot B^s_{p,q}}\approx \left( \int_0^\infty \left(\frac{\norm{\beta_{(n)}(f, I(\cdot, t))}_{L^p}}{t^{s-1}}\right)^q \frac{dt}{t}\right)^{1/q}.$$
In the particular case when $p=q$, which is in fact the one we are interested on, it is enough to consider dyadic intervals. Namely, writing $\mathcal{D}$ for the canonical dyadic grid, via Fubini's Theorem one can conclude that
\begin{equation}\label{eqNormBetas}
\norm{f}_{\dot B^s_{p,p}}^p\approx \sum_{I\in\mathcal{D}}\left(\frac{\beta_{(n)}(f,I)}{|I|^{s-1}}\right)^p |I|.
\end{equation}
When restricting to an open interval $I$, we call 
\begin{equation}\label{eqBesovRestricted}
\norm{f}_{\dot B^s_{p,p}(I)}^p:= \inf_{F: F|_I \equiv f}\norm{F}_{\dot B^s_{p,p}(I)}.
\end{equation}

Consider the boundary of a Lipschitz domain $\Omega \subset \C$. When it comes to the Besov space $B^s_{p,q}(\partial \Omega)$ we can just define it using the arc parameter of the curve, $z:I \to \partial\Omega$ with $|z'(t)|=1$ for all $t$. Note that if the domain is bounded, then $I$ is a finite interval with length equal to the length of the boundary of $\Omega$ and we need to extend $z$ periodically to $\R$ in order to have a sensible definition.
Then, if $1\leq p,q< \infty$,  we define naturally the homogeneous Besov norm on the boundary of $\Omega$ as
\begin{equation*}
\norm{f}_{B^s_{p,q}(\partial\Omega)} :=\norm{f\circ z}_{L^p(I)}+\norm{f\circ z}_{\dot B^s_{p,q}(2I)}.
\end{equation*}

Let $n\geq 1$, $\delta, R>0$ and  let $\Omega$ be a bounded $(\delta,R)-C^{n-1,1}$ domain. Consider $N:\partial\Omega\to \R^2$ to be the unitary outward normal vector of a Lipschitz domain. The following lemma gives a relation between the Besov norm of $N$ and the Betas of the parameterizations of the boundary of the domain. For this we will ask to have some controlled overlapping of the windows that we consider. 
 \begin{lemma}\label{lemNormBeta}
Let $n\geq 1$, $\delta, R>0$,  let $\Omega$ be a bounded $(\delta,R)-C^{n-1,1}$ domain and let $\{\mathcal{Q}_k\}_{k=1}^M$ be a collection of $R$-windows  such that $\left\{\frac{1}{20} \mathcal{Q}_k\right\}_{k}$ cover the boundary of $\Omega$ and $\left\{\frac{1}{40}  \mathcal{Q}_k\right\}_{k}$ are disjoint. Let $\{A_k\}_k$ be the parameterizations of the boundary associated to each window. Then, for any $1<p<\infty$
\begin{equation*}
\sum_{k=1}^M\sum_{I\in\mathcal{D}: I\subset \frac16 I_{R}}\frac{\beta_{(n)}( A_k,I)^p}{\ell(I)^{n\,p-2}} \lesssim \sum_{k=1}^M\norm{A_k}_{\dot B^{n+1-1/p}_{p,p}(\frac13 I_{R})}^p \lesssim  \norm{N}_{B^{n-1/p}_{p,p}(\partial\Omega)}^p,
\end{equation*}
where $I_R$ stands for the interval $(-R,R)$. The constants depend on $n$, $p$, $\delta$, $R$ and the length of the boundary $\mathcal{H}^1(\partial\Omega)$.
\end{lemma}
The proof of this lemma for $n=1$ can be found in \cite[Lemma 3.3]{CruzTolsa}. The  case $n\geq2$ is quite technical but uses the same tools, its proof can be found in the appendix.

\subsection{A family of convolution operators in the plane}\label{secOperators}
\begin{definition}
Consider a function $K:\C \setminus\{0\} \to \C$. For any $f\in L^1_{loc}$ we define
$$T^K f(z)=\lim_{\varepsilon\to 0}\int_{\C \setminus B_\varepsilon(z)}K(z-w)f(w) \,dm(w)$$
as long as the limit exists, for instance, when $K$ is bounded away from $0$, $f\in L^1$ and $z\notin \supp(f)$ or when $f=\chi_U$ for an open set $U$ with $z\in U$, $\int_{B_\varepsilon(0)\setminus B_{\varepsilon'}(0)} K \, dm =0$ for every $\varepsilon>\varepsilon'>0$ and $K$ is integrable at infinity. We say that $K$ is the kernel of $T^K$. 

For any multiindex $\gamma \in \Z^2$, we will consider $K^\gamma(z)=z^{\gamma}=z^{\gamma_1}\bar{z}^{\gamma_2}$ and then we will put shortly $T^\gamma f:=T^{K^\gamma} f$, that is,
\begin{equation}\label{eqTgamma}
T^\gamma f(z)=\lim_{\varepsilon\to 0}\int_{\C \setminus B_\varepsilon(z)}(z-w)^\gamma f(w) \,dm(w)
\end{equation}
as long as the limit exists.

For any operator $T$ and any domain $\Omega$, we can consider $T_\Omega f= \chi_\Omega \, T(\chi_\Omega\, f)$.
\end{definition}

\begin{example}
As the reader may have observed, the Beurling transform is in that family of operators. Namely, when $K(z)=z^{-2}$, that is, for $\gamma=(-2,0)$, then $\frac{-1}{\pi}T^\gamma$ is the Beurling transform. The operator $\frac1\pi T^{(-1,0)}$  is the so-called Cauchy transform which we denote by $\Cauchy$.

Consider the iterates of the Beurling transform $\Beurling^m$ for $m>0$. For every $f \in L^p$ and $z\in \C$ we have
\begin{align}\label{eqIterateBeurling}
\Beurling^mf(z)
	& =\frac{(-1)^m m}{\pi}\lim_{\varepsilon\to 0}\int_{|z-\tau|>\varepsilon}\frac{(\overline{z-\tau})^{m-1}}{(z-\tau)^{m+1}}f(\tau)\, dm(\tau) =\frac{(-1)^m m}{\pi}T^{(-m-1,m-1)}f(z)
\end{align}
(see \cite[Section 4.2]{AstalaIwaniecMartin}). That is, for $\gamma=(\gamma_1, \gamma_2)$ with $\gamma_1+\gamma_2=-2$ and $\gamma_1\leq -2$, the operator $T^\gamma$ is an iteration of the Beurling transform modulo constant, and it maps $L^p(U)$ to itself for every open set $U$. If $\gamma_2\leq -2$ instead, then $T^\gamma$ is an iterate of the conjugate Beurling transform and it is bounded in $L^p$ as well.
 \end{example}

\section{The characteristic function}\label{secCharacteristic}
\subsection{The case of unbounded domains $\Omega\subset \C$}\label{secCharacteristicSpecial}
\begin{definition}\label{defAdmissible}
Given $n \in\N$, $1<p<\infty$, $\delta>0$ and $R>0$, we say that $\Omega=\{x+i\,y\in\C: y>A(x)\}$ is a $(\delta,R,n,p)$-admissible domain with defining function $A$ if 
\begin{itemize}
\item the defining function $A\in B^{n+1-1/p}_{p,p}\cap C^{n-1,1}$,
%\item Moreover, $A$ is supported on the interval $[-R,R]$.
\item we have $A(0)=0$ and, if $n\geq 2$,  $A'(0)=0$,
\item and we have Lipschitz bounds on the function and its derivatives $\norm{A^{(j)}}_{L^\infty}<\frac{\delta}{R^{j-1}}$ for $1\leq j\leq n$.
\end{itemize}

We associate a Whitney covering $\mathcal{W}$ with appropriate constants to $\Omega$. The constants will be fixed along this section, depending on $n$ and $\delta$.
\end{definition}

In this Section we will prove the next result for the operators $T^\gamma$ defined in \rf{eqTgamma}.
\begin{theorem}\label{mtheorem}
Consider  $\delta,R,\epsilon>0$, $p>1$ and a natural number $n\geq1$. There exists a radius $\rho_\epsilon<R$ such that for every $(\delta,R,n,p)$-admissible domain
 $\Omega$ and every multiindex $\gamma \in \Z^2$ with $\gamma_1+\gamma_2=-n-2$ and $\gamma_1\cdot \gamma_2\leq 0$, we have that $T^\gamma \chi_\Omega \in L^p(\Omega\cap B(0,\rho_\epsilon))$ and, if $A$ is the defining function of $\Omega$, then the estimate
$$\norm{T^\gamma\chi_\Omega}_{L^p(\Omega\cap B(0, \rho_\epsilon))}^p\leq C \left( \norm{A}^p_{\dot B_{p,p}^{n-1/p+1}(-5\rho_\epsilon,5\rho_\epsilon)}+  \rho_\epsilon^{2-np}(1+\epsilon)^{|\gamma|p}\right)$$
is satisfied, where  $C$ depends on $p$, $n$ and the Lipschitz character of $\Omega$ (see Figure \ref{figUnboundedAdmissible}). 
\end{theorem}

\begin{figure}[ht]
 \centering
 {\includegraphics[width=0.6\textwidth]{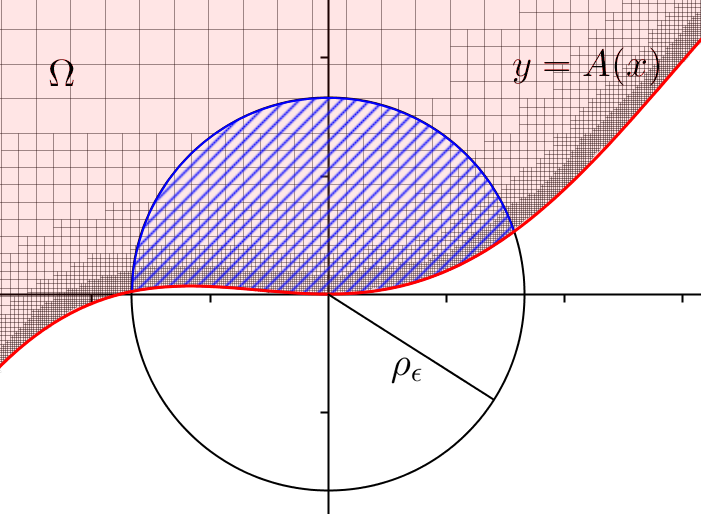}}    
  \caption{Disposition in Theorem \ref{mtheorem}.}\label{figUnboundedAdmissible}
\end{figure}

\begin{definition}
Consider $n \in\N$, $1<p<\infty$, $\delta>0$, $R>0$ and a $(\delta,R,n,p)$-admissible domain with defining function $A$. Then, for every interval $I$ we have an approximating polynomial $\mathbf{R}^n_{3I}:=\mathbf{R}^n_{3I}A$, and
$$\beta_{(n)}(I):=\frac{1}{\ell(I)}\int_{3I}\frac{|A(x)-\mathbf{R}^n_{3I}(x)|}{\ell(I)}\, dx.$$
We call 
$$\Omega^n_I:=\{x+i\,y: y>\mathbf{R}^n_{3I}(x)\}.$$
Let $\pi: \C \to \R$ be the vertical projection (to the real axis) and $Q$ a cube in $\C$. If $\pi(Q)=I$ we will write $\Omega^n_Q:=\Omega^n_I$.
\end{definition}

\begin{remark}
 Note that $\pi$ sends dyadic cubes of $\C$ to dyadic intervals of $\R$ and, in particular, any dyadic interval has a finite number of pre-images in the Whitney covering $\mathcal{W}$ of $\Omega$ uniformly bounded by a constant depending on $\delta$ and the Whitney constants of $\mathcal{W}$.
 \end{remark}

\begin{proof}[Proof of Theorem \ref{mtheorem}]
By \rf{eqNormBetas} we have that $\sum_{I\in \mathcal{D}}\left(\frac{\beta_{(n)}(I)}{\ell(I)^{n-1/p}}\right)^p\ell(I) \approx \norm{A}^p_{\dot B_{p,p}^{n-1/p+1}},$
and, by \rf{eqBesovRestricted} we get
$$\sum_{I\in \mathcal{D}_\epsilon}\left(\frac{\beta_{(n)}(I)}{\ell(I)^{n-1/p}}\right)^p\ell(I) \lesssim \norm{A}^p_{\dot B_{p,p}^{n-1/p+1}(-5\rho_\epsilon,5\rho_\epsilon)},$$
where $\mathcal{D}_\epsilon$ stands for $\{I\in\mathcal{D}: \ell(I)\leq 2\rho_\epsilon \mbox{ and } I\subset (-3\rho_\epsilon, 3\rho_\epsilon)\}$.
Thus, it is enough to prove that 
\begin{equation}\label{eqTargetNorm}
\norm{T^{\gamma}\chi_\Omega}_{L^p(\Omega\cap B(0, \rho_\epsilon))}^p\leq C \left( \sum_{I\in \mathcal{D}_\epsilon}\left(\frac{\beta_{(n)}(I)}{\ell(I)^{n-1/p}}\right)^p\ell(I) + \rho_\epsilon^{2-np}(1+\epsilon)^{|\gamma|p}\right).
\end{equation}

We begin the proof by some basic observations. Let $j_1,j_2\in\Z$ such that $j_2\neq j_1+1$. Then, the line integral
\begin{equation}\label{eqCancellation}
\int_{\partial\DDD} w^{j_1} \overline{w}^{j_2} \, dw = i \int_0^{2\pi} e^{i\theta(j_1-j_2+1)}d\theta = 0.
\end{equation}
If, moreover, $j_2>0$, given $0<\varepsilon<1$ Green's formula \rf{eqGreen} says that
\begin{equation}\label{eqCancellationGreen}
\int_{\DDD\setminus B(0,\varepsilon)} w^{j_1} \overline{w}^{j_2-1} \, dm(w)=\frac{i}{2 j_2} \left(\int_{\partial\DDD} - \int_{\partial B(0,\varepsilon)}\right) w^{j_1} \overline{w}^{j_2} \, dw=0.
\end{equation}

Consider a given $\gamma\in\Z^2$ with $\gamma_1+\gamma_2=-n-2$ and assume that $\gamma_2\geq0$ (the case $\gamma_1\geq 0$ can be proven mutatis mutandis). Consider a Whitney cube $Q$ and $z\in B(0,\rho_\epsilon)\cap Q$. Then by \rf{eqCancellationGreen} we have that
\begin{align}\label{eqDosSumands}
|T^{\gamma}\chi_\Omega(z)|
	& =  \left|\int_{|z-w|> \ell(Q)} (w-z)^{\gamma}\chi_\Omega(w) \, dm(w) \right| \\
\nonumber	& \leq  \left|\int_{|z-w|> \ell(Q)} (w-z)^{\gamma}\chi_{\Omega^n_Q}(w) \, dm(w) \right| + \int_{|z-w|> \ell(Q)} \frac{|\chi_{\Omega^n_Q}(w)-\chi_\Omega(w)|}{|w-z|^{n+2}} \, dm(w) .
\end{align}
If we have taken appropriate Whitney constants, then we also have that $\ell(Q)<\dist(Q,\partial\Omega^n_Q)$ (see  Remark \ref{remApproxPoly}) and, thus, by \rf{eqCancellationGreen} again, we have that 
\begin{equation}\label{eqBackToOperator}
 \int_{|z-w|> \ell(Q)} (w-z)^{\gamma}\chi_{\Omega^n_Q}(w) \, dm(w) = T^{\gamma}\chi_{\Omega^n_Q}(z).
 \end{equation}
We will see in Section \ref{secPoly} that the following claim holds.

\begin{claim}\label{claimPoly} There exists a radius $\rho_\epsilon$ (depending on $\delta$, $R$, $n$ and $\epsilon$) such that for every  $z\in B(0,\rho_\epsilon)$ with $z\in Q\in\mathcal{W}$, we have that
\begin{equation}\label{eqBoundOperator}
|T^{\gamma}\chi_{\Omega^n_Q}(z)|\lesssim_{n} \frac{(1+\epsilon)^{|\gamma|}}{\rho_\epsilon^n}.
\end{equation}
\end{claim}

The last term in \rf{eqDosSumands} will bring the beta coefficients into play. Recall that we defined the symmetric difference of two sets $A_1$ and $A_2$ as $A_1\Delta A_2:=(A_1\cup A_2) \setminus (A_1\cap A_2)$. Our choice of the Whitney constants can grant that $3Q\subset \Omega^n_Q \cap \Omega$ so
\begin{align}\label{eqAbsoluteInside}
\int_{|z-w|> \ell(Q)} \frac{|\chi_{\Omega^n_Q}(w)-\chi_\Omega(w)|}{|w-z|^{n+2}} \, dm(w)
	& =\int_{\Omega^n_Q \Delta \Omega} \frac{1}{|w-z|^{n+2}} \, dm(w).
\end{align}
Next we split the domain of integration in vertical strips. Namely, if we call $S_j=\{w\in\C :  |\real (w-z)|\leq2^j \ell(Q)\}$ for $j\geq0$ and $S_{-1}=\emptyset$, we have that
\begin{align}\label{eqAnulli}
\int_{\Omega^n_Q \Delta \Omega} \frac{1}{|w-z|^{n+2}} \, dm(w)
\nonumber	& =\sum_{j\geq 0: \, 2^j\ell(Q) \leq \rho_\epsilon} \int_{(\Omega^n_Q \Delta \Omega) \cap S_j\setminus S_{j-1}} \frac{dm(w)}{|w-z|^{n+2}}+\int_{|w-z|>\rho_\epsilon/2} \frac{dm(w)}{|w-z|^{n+2}}\\
	& \lesssim \sum_{j\geq 0: \, 2^j\ell(Q) \leq \rho_\epsilon} \left|(\Omega^n_Q \Delta \Omega) \cap S_j \right| \frac{1}{(2^{j-1}\ell(Q))^{n+2}} +\frac{1}{\rho_\epsilon^n}.
\end{align}
We will see in Section \ref{secInterstitial} the following:

\begin{claim}\label{claimInter} We have that
\begin{equation}\label{eqBoundInter}
\left|(\Omega^n_Q \Delta \Omega) \cap S_j \right| \lesssim_{n} \sum_{\substack{I\in \mathcal{D}\\ \pi(Q)\subset I\subset 2^{j+1} \pi(Q)}} \frac{\beta_{(n)}(I)}{\ell(I)^{n-1}} (2^j \ell(Q))^{n+1}.
\end{equation}
\end{claim}

Summing up, plugging \rf{eqBackToOperator} and \rf{eqBoundOperator} in the first term of the right-hand side of \rf{eqDosSumands} and plugging \rf{eqAbsoluteInside}, \rf{eqAnulli} and \rf{eqBoundInter} in the other term, we get
$$|T^{\gamma}\chi_\Omega(z)| \lesssim_{n} \sum_{\substack{j\geq 0\\ 2^j\ell(Q) \leq \rho_\epsilon}} \sum_{\substack{I\in \mathcal{D}\\ \pi(Q)\subset I\subset 2^{j+1} \pi(Q)}} \frac{\beta_{(n)}(I)}{\ell(I)^{n-1}} (2^j \ell(Q))^{n+1} \frac{1}{(2^{j}\ell(Q))^{n+2}} + \frac{(1+\epsilon)^{|\gamma|}}{\rho_\epsilon^n}.$$
Note that the intervals $I$ in the previous sum are in $\mathcal{D}_\epsilon=\{I\in\mathcal{D}: \ell(I)\leq 2\rho_\epsilon \mbox{ and } I\subset (-3\rho_\epsilon, 3\rho_\epsilon)\}$. Reordering and computing, 
\begin{align*}
|T^{\gamma}\chi_\Omega(z)|
	& \lesssim_{n}  \sum_{\substack{I\in \mathcal{D}_\epsilon\\  \pi(Q)\subset I}} \frac{\beta_{(n)}(I)}{\ell(I)^{n-1}} \sum_{\substack{j\in \N_0\\ I\subset 2^{j+1}
 \pi(Q)}}\frac{1}{2^{j}\ell(Q)} +  \frac{(1+\epsilon)^{|\gamma|}}{\rho_\epsilon^n} \lesssim \sum_{\substack{I\in \mathcal{D}_\epsilon\\  \pi(Q)\subset I}} \frac{\beta_{(n)}(I)}{\ell(I)^{n}} + \frac{(1+\epsilon)^{|\gamma|}}{\rho_\epsilon^n}.
\end{align*}

Raising to power $p$, integrating in $Q$ and adding we get that for $\rho_\epsilon$ small enough 
\begin{align}\label{eqAlmostDone}
\norm{T^{\gamma}\chi_\Omega}_{L^p(\Omega\cap B(0,\rho_\epsilon))}^p 
\nonumber	& \lesssim_{n} \sum_{\substack{Q\in\mathcal{W}\\ Q \cap B(0,\rho_\epsilon)\neq \emptyset}} |Q| \left( \sum_{\substack{I\in \mathcal{D}_\epsilon\\  \pi(Q)\subset I}} \frac{\beta_{(n)}(I)}{\ell(I)^{n}} + \frac{(1+\epsilon)^{|\gamma|}}{\rho_\epsilon^n} \right)^p  \\
			& \lesssim_p \sum_{\substack{Q\in\mathcal{W}\\ Q \cap B(0,\rho_\epsilon)\neq \emptyset}} |Q| \left( \sum_{\substack{I\in \mathcal{D}_\epsilon\\  \pi(Q)\subset I}} \frac{\beta_{(n)}(I)}{\ell(I)^{n}} \right)^p + \rho_\epsilon^{2-np}(1+\epsilon)^{|\gamma|p}.
	\end{align}
Regarding the double sum, we use H\"older's Inequality to find that
\begin{align}\label{eqDone}
 \sum_{\substack{Q\in\mathcal{W}\\  Q \cap B(0,\rho_\epsilon)\neq \emptyset}} |Q| \left( \sum_{\substack{I\in \mathcal{D}_\epsilon\\  \pi(Q)\subset I}} \frac{\beta_{(n)}(I)}{\ell(I)^{n}} \right)^p
\nonumber 	& \leq \sum_{Q\in\mathcal{W}} |Q|  \sum_{\substack{I\in \mathcal{D}_\epsilon\\  \pi(Q)\subset I}} \left(\frac{\beta_{(n)}(I)}{\ell(I)^{n-\frac1{2p}}}\right)^p \left( \sum_{\substack{I\in \mathcal{D}_\epsilon\\  \pi(Q)\subset I}} \frac{1}{\ell(I)^{\frac{p'}{2p}}}\right)^\frac{p}{p'}\\
			 & \lesssim_p \sum_{Q\in\mathcal{W}} \ell(Q)^2  \sum_{\substack{I\in \mathcal{D}_\epsilon\\  \pi(Q)\subset I}} \left(\frac{\beta_{(n)}(I)}{\ell(I)^{n-\frac1{2p}}}\right)^p \ell(Q)^\frac{-1}2 \\
\nonumber	 & \leq \sum_{I\in \mathcal{D}_\epsilon} \left(\frac{\beta_{(n)}(I)}{\ell(I)^{n-\frac1{2p}}}\right)^p \sum_{\substack{Q\in\mathcal{W}\\  \pi(Q)\subset I}}\ell(Q)^\frac32 \lesssim_{\mathcal{W}} \sum_{I\in \mathcal{D}_\epsilon} \left(\frac{\beta_{(n)}(I)}{\ell(I)^{n-\frac1{p}}}\right)^p \ell(I),
\end{align}
where the constant in the last inequality depends on the maximum number of Whitney cubes that can be projected to a given interval, depending only on $\delta$ and $n$.

Thus, by  \rf{eqAlmostDone} and \rf{eqDone} we have proven \rf{eqTargetNorm} when $\gamma_2\geq 0$. The case $\gamma_2\leq 0$ can be proven analogously.
\end{proof}

%\begin{remark}
%By Lemma \ref{lemNormBeta} the last sum is the Besov seminorm $\norm{A}^p_{\dot B^s_{p,p}}$ for smoothness $s=n+1-1/p$, which is bounded by  the norm $\norm{N}^p_{B^{s-1}_{p,p}(\partial\Omega)}$ with constants depending on the $C^{n-1,1}$ constants of the domain and the length of its boundary. Thus, equations \rf{eqDerivadaConstants}, \rf{eqAlmostDone} and \rf{eqDone} lead to
%$$\norm{D^\alpha T^{\gamma}\chi_\Omega}_{L^p(\Omega\cap B(0,\rho_\epsilon))}^p 
% \lesssim |\gamma|^{np} \left(\norm{N}_{B^{n-1/p}_{p,p}(\partial\Omega)}^p+  \rho_\epsilon^{2-np}(1+\epsilon)^{|\gamma|}\right).$$
%\end{remark}

\subsection{The interstitial region}\label{secInterstitial}
\begin{proof}[Proof of Claim \ref{claimInter}]
\begin{figure}[ht]
 \centering
 {\includegraphics[width=\textwidth]{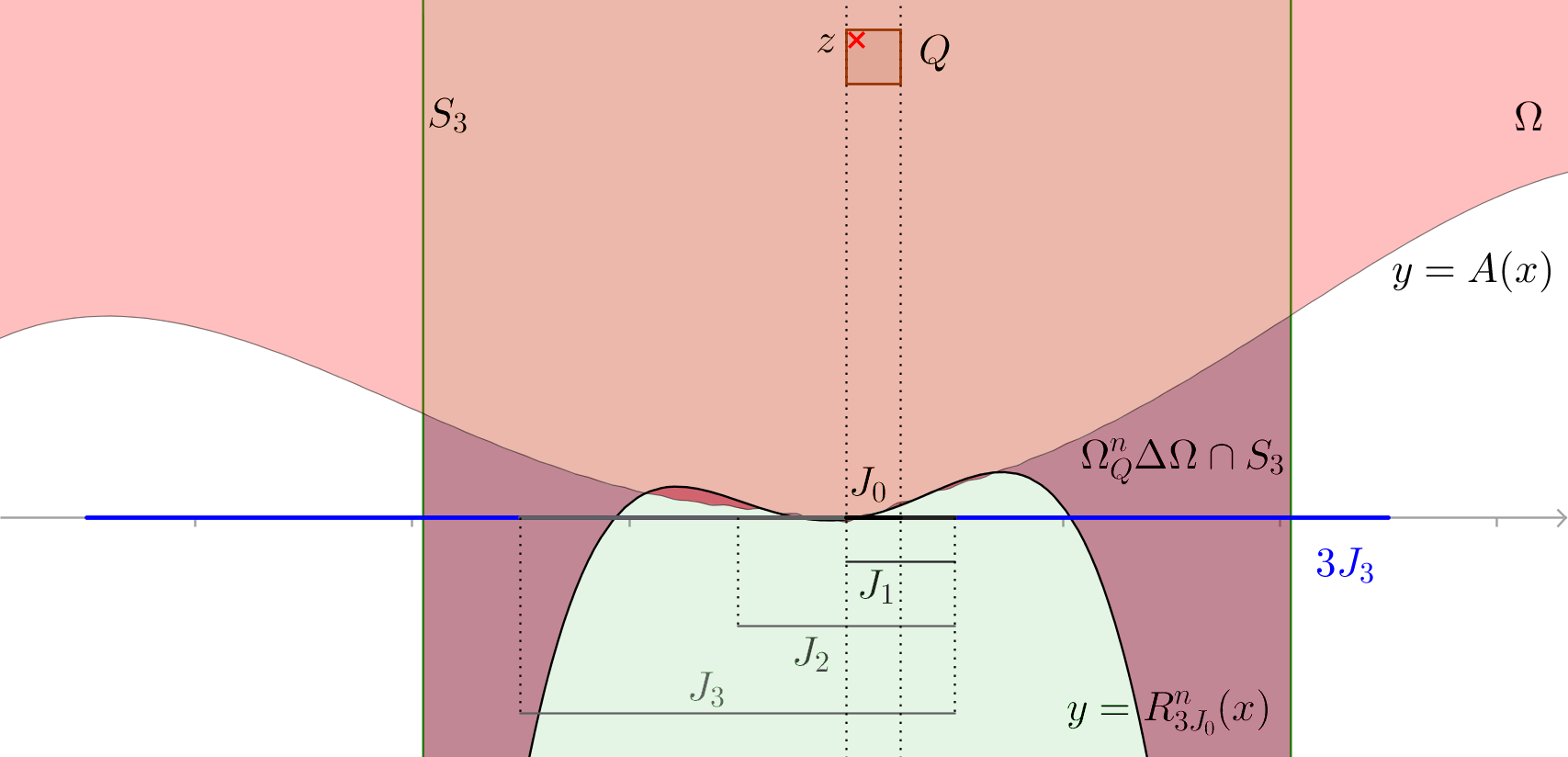}}    
  \caption{Disposition in the proof of Claim \ref{claimInter} for $N=3$.}\label{figBetaStepIn}
\end{figure}

Consider $N\geq 0$. Recall that we have a point $z\in Q\in\mathcal{W}$, and a vertical strip $S_N=\{w\in\C :  |\real (w-z)|\leq2^N \ell(Q)\}$. Let $J_0=\pi(Q)$ and let $J_{N}$ be the dyadic interval of length $2^{N}\ell(Q)$ containing $J_0$ (see Figure \ref{figBetaStepIn}). Then it is enough to see that
\begin{equation}\label{eqBoundByBetas}
\left|(\Omega^n_{Q} \Delta \Omega) \cap S_N \right| \lesssim_n \sum_{\substack{I\in \mathcal{D}\\ J_0 \subset I\subset J_N}} \beta_{(n)}(I)\frac{\ell(J_N)^{n-1}}{\ell(I)^{n-1}} \ell(J_N)^2.
\end{equation}

First note that
\begin{align}\label{eqFirstDecomposition}
\left|(\Omega^n_{Q} \Delta \Omega) \cap S_N \right|
	& = \int_{\real(z)-\ell(J_N)}^{\real(z)+\ell(J_N)}|A-\mathbf{R}^n_{3J_0}|\,dm_1 \\
\nonumber	& \leq \int_{3J_N}|A-\mathbf{R}^n_{3J_N}| \, dm_1 + \int_{3J_N}|\mathbf{R}^n_{3J_N}-\mathbf{R}^n_{3J_0}| \, dm_1=\squared{1}+\squared{2}.
\end{align}
Trivially, 
\begin{equation}\label{eqPart1}
\squared{1}=\beta_{(n)}(J_N)\ell(J_N)^2.
\end{equation}

To deal with the second term, we consider the chain of dyadic intervals
$$J_0\subset\cdots\subset J_k\subset J_{k+1}\subset \cdots \subset J_N,$$
with $0<k<N$ and $\ell(J_k)=2^k\ell(J_0)$.  We use the Triangle Inequality in the chain of intervals:
\begin{equation}\label{eqChainDecomposition}
\squared{2}\leq\sum_{k=0}^{N-1}\int_{3J_N} |\mathbf{R}^n_{3J_{k+1}}-\mathbf{R}^n_{3J_k}|\, dm_1 = \sum_{k=0}^{N-1}\norm{\mathbf{R}^n_{3J_{k+1}}-\mathbf{R}^n_{3J_k}}_{L^1(3J_N)}.
\end{equation}

For any polynomial $P(x)=\sum_{i=1}^n a_i x^i$ of degree $n$ and any interval $J$ centered at $0$, using the linear map $\phi$ that sends the interval $(-1,1)$ to $J$ as a change of coordinates, we have that
$$\norm{P}_{L^1(J)}\approx \ell(J) \norm{P\circ \phi}_{L^1(-1,1)},$$
and using the fact that all norms in a finite dimensional vector space are equivalent (in particular the $L^1(-1,1)$ norm and the sum of coefficients) we have that 
$$\norm{P}_{L^1(J)}\approx_n \ell(J) \sum_{i=1}^n \ell(J)^i\left| a_i\right|.$$ 
By the same token, for any $k_0\in \N$, we get
$$\norm{P}_{L^1(2^{k_0} J)}\approx_n 2^{k_0} \ell(J) \sum_{i=1}^n \left(2^{k_0} \ell(J)\right)^i\left| a_i\right|\lesssim_n 2^{k_0 (n+1)} \norm{P}_{L^1(J)} .$$ 
Fix $0\leq k<N$. Then
$$\norm{\mathbf{R}^n_{3J_{k+1}}-\mathbf{R}^n_{3J_k}}_{L^1(3J_N)}\lesssim_n\norm{\mathbf{R}^n_{3J_{k+1}}-\mathbf{R}^n_{3J_k}}_{L^1(3J_k)}\frac{\ell(J_N)^{n+1}}{\ell(J_k)^{n+1}},$$
with constants depending only on $n$. Thus, we have that
\begin{align}\label{eqBetasIntoPlay}
\norm{\mathbf{R}^n_{3J_{k+1}}-\mathbf{R}^n_{3J_k}}_{L^1(3J_N)}
\nonumber	& \lesssim_n \left(\norm{\mathbf{R}^n_{3J_{k+1}}-A}_{L^1(3J_k)}+\norm{A-\mathbf{R}^n_{3J_k}}_{L^1(3J_k)}\right)\frac{\ell(J_N)^{n+1}}{\ell(J_k)^{n+1}}\\
		& \lesssim_n \left( \beta_{(n)}(J_{k+1})+\beta_{(n)}(J_k)\right)\frac{\ell(J_N)^{n+1}}{\ell(J_k)^{n+1}}\ell(J_k)^2.
\end{align}

Combining \rf{eqFirstDecomposition}, \rf{eqPart1}, \rf{eqChainDecomposition} and \rf{eqBetasIntoPlay} we get \rf{eqBoundByBetas}.
\end{proof}

\subsection{Domain bounded by a polynomial graph}\label{secPoly}

We will consider only very ``flat'' polynomials. Let us see what we can say about their coefficients.

\begin{lemma}
Let $n\geq 2$, $A\in C^{n-1,1}(\R)$  with $A(0)=0$, $A'(0)=0$, $\norm{A^{(j)}}_{L^\infty}<\frac{\delta}{R^{j-1}}$ for $j\leq n$ and consider two intervals $J$ and $I$ with $3J\subset I=[-R,R]$. Then we have the following bounds for the derivatives of the approximating polynomial $P=\mathbf{R}^n_JA$ in the interval $I$:
$$\norm{P^{(j)}}_{L^\infty(I)}\leq \frac{3^{n-j} \delta}{R^{j-1}} \mbox{\,\,\,\, for } j\leq n.$$
Furthermore, if $\rho>0$ and $3J\subset [-\rho,\rho]$, then 
\begin{equation}\label{eqNormPolinomialBetter}
\norm{P}_{L^\infty(-\rho,\rho)}\leq \frac{3^n\delta \rho^2 }R\mbox{\quad\quad and \quad\quad} \norm{P'}_{L^\infty(-\rho,\rho)}\leq \frac{3^{n-1} \delta \rho}R.
\end{equation}
\end{lemma}
\begin{proof}
By Remark \ref{remCut} we know that there are at least $n+1$ common points $\tau_0^0, \cdots,\tau_n^0\in 3J$ for $A$ and $P$, that is, $A(\tau_j^0)=P(\tau_j^0)$ for every $j$. By the Mean Value Theorem, there are $n$ common points $\tau_0^1, \cdots, \tau_{n-1}^1\in 3J$ for their derivatives. By induction we find points $\tau_0^{k}\cdots\tau_{n-k}^{k}\in 3J$ where the $k$-th derivatives coincide for $0\leq k \leq n-1$, that is, $A^{(k)}(\tau_j^k)=P^{(k)}(\tau_j^k)$ for every  $0\leq j \leq n-k$.

Note that the polynomial derivative $P^{(n)}$, which is in fact a constant, coincides with the differential quotient of $P^{(n-1)}$ evaluated at any pair of points. In particular given $x\in \R$, for the points $\tau_0^{n-1}$ and $\tau_1^{n-1}$ we have that
$$\left|P^{(n)}(x)\right| = \left|\frac{P^{(n-1)}(\tau_0^{n-1})-P^{(n-1)}(\tau_1^{n-1})}{\tau_0^{n-1}-\tau_1^{n-1}}\right|
	= \left|\frac{A^{(n-1)}(\tau_0^{n-1})-A^{(n-1)}(\tau_1^{n-1})}{\tau_0^{n-1}-\tau_1^{n-1}}\right|\leq \frac{\delta}{R^{n-1}}.$$

Now we argue by induction again. Assume that $\norm{P^{(j+1)}}_{L^\infty(I)}\leq 3^{n-j-1}\delta/R^{j}$ for a certain $j\leq n-1$. Consider $x\in I$ and, by the Mean Value Theorem, there exists a point $\xi$ such that $ |P^{(j)}(x)-P^{(j)}(\tau^j_0)|=|P^{(j+1)}(\xi)||x-\tau^j_0|$. Thus, since $P^{(j)}(\tau^j_0)=A^{(j)}(\tau^j_0)$ we have that
\begin{align*}
|P^{(j)}(x)|
	 & \leq |P^{(j+1)}(\xi)||x-\tau^j_0|+|A^{(j)}(\tau^j_0)| \leq \frac{3^{n-j-1}\delta}{R^j}\,2R+\frac{\delta}{R^{j-1}}=\frac{3^{n-j}\delta}{R^{j-1}}.
\end{align*}

We have not used yet the fact that $A'(0)=A(0)=0$. Let us fix $\rho\leq R$ and assume that $3J\subset [-\rho,\rho]$. Then for every $x\in [-\rho,\rho]$, we can write $A'(x)=A'(x)-A'(0)$ so
\begin{equation}\label{eqEstimateA'}
|A'(x)|\leq \norm{A''}_{L^\infty(I)}|x| \leq    \frac{\delta}{R}\, \rho,
\end{equation}
and we can also write $P'(x)= P'(x)- P'( \tau_0^1) + A'( \tau_0^1)-A'(0)$, so
$$|P'(x)|\leq \norm{P''}_{L^\infty(I)}|x-\tau_0^1| + \norm{A''}_{L^\infty(I)}|\tau_0^1| \leq \frac{3^{n-2}\delta}{R}\, 2 \rho +  \frac{\delta}{R}\, \rho \leq \frac{3^{n-1}\delta\rho}{R}.$$
By the same token, and using the estimate \rf{eqEstimateA'} on $A'$, we get
$$|P(x)| \leq \norm{P'}_{L^\infty([-\rho,\rho])}|x-\tau_0^0| + \norm{A'}_{L^\infty([-\rho,\rho])}|\tau_0^0| \leq \frac{3^{n-1}\delta\rho}{R} \, 2 \rho + \frac{\delta\rho}{R} \, \rho \leq \frac{3^{n}\delta\rho^2}{R}.$$
\end{proof}

Now we can prove Claim \ref{claimPoly}. Recall that we want to find a radius $\rho_{int}<R$ depending on $\epsilon$ such that every point $z$ contained in a  Whitney cube $Q \subset B(0,\frac{\rho_{int}}{2})$ satisfies \rf{eqBoundOperator}, that is, 
\begin{equation*}
|T^{\gamma}\chi_{\Omega^n_Q}(z)|\lesssim_{n} \frac{(1+\epsilon)^{|\gamma|}}{\rho_{int}^n},
\end{equation*}
where $\gamma \in \{(-j_1,j_2): j_1,j_2\in \N_0 \mbox{ and }  j_1-j_2=n+2\}$ (recall that we assumed that $\gamma_2\geq 0$). 
According to the previous lemma, when $n\geq2$  we are dealing with a domain $\Omega^n_Q$ whose boundary is the graph of a polynomial $P(x)=\sum_{j=0}^n a_j x^j$ such that
\begin{align}\label{eqBoundCoefficients}
\nonumber	|a_0|	=|P(0)|		&\leq \frac{3^n \delta \rho_{int}^2}{R}, &&\\
\nonumber	|a_1|	=|P'(0)|		&\leq \frac{3^{n-1}\delta\rho_{int}}{R} &&\mbox{and}\\
			|a_j|	=\frac{|P^{(j)}(0)|}{j!}	&\leq \frac{3^{n-j} \delta}{j!R^{j-1}} && \mbox{for }2\leq j <n.
\end{align}
We call $\Omega_P:=\{ x+i\,y : y>P(x)\}$ to such a domain. Note that \rf{eqNormPolinomialBetter} implies that for $\rho_{int}$ small enough the polynomial $P$ is ``flat'', namely $|P(x)|<\frac{\rho_{int}}4$ for $|x|<\rho_{int}$.

One can think of the ``exterior'' radius $\rho_{ext}$ below as a geometric version of $\epsilon$, namely $\rho_{ext}=(\epsilon/16)^2$ if $\epsilon$ is small enough. Further, we can assume that $\rho_{ext}<R$.

\begin{proposition}\label{propoPoly}
Consider two real numbers $\delta, R>0$ and $n \geq 2$. For $\rho_{ext}$ small enough, there exists $0<\rho_{int}<\rho_{ext}$ depending also on $n$,  $\delta$ and $R$ such that for all $j_1, j_2 \in\N_0$ with $j_1-j_2=n+2$, all $P\in\mathcal{P}^n$ satisfying \rf{eqBoundCoefficients}, all $z\in Q(0, \rho_{int})\cap  \Omega_P$ and $0<\varepsilon<\dist(z, \partial\Omega_P)$ we have
\begin{equation}\label{eqPolyIntegralBounded}
\left| \int_{\Omega_P\setminus B(z,\varepsilon)} \frac{(\overline{z-w})^{j_2}}{(z-w)^{j_1}}\, dm(w)\right|\leq \frac{C_n}{\rho_{int}^{n}}\left(1+16\rho_{ext}^{1/2}\right)^{j_2},
\end{equation} 
 with $C_n$ depending only on $n$.
 
  If $n=1$ instead, then for all $j_1, j_2 \in\N_0$ with $j_1-j_2=3$ and all $P\in\mathcal{P}^1$ we have that
 \begin{equation*}
 \int_{\Omega_P\setminus B(z,\varepsilon)} \frac{(\overline{z-w})^{j_2}}{(z-w)^{j_1}}\, dm(w) =0.
\end{equation*} 
\end{proposition}

\begin{proof}
\begin{figure}[ht]
  \centering
 {\includegraphics[width=0.6\textwidth]{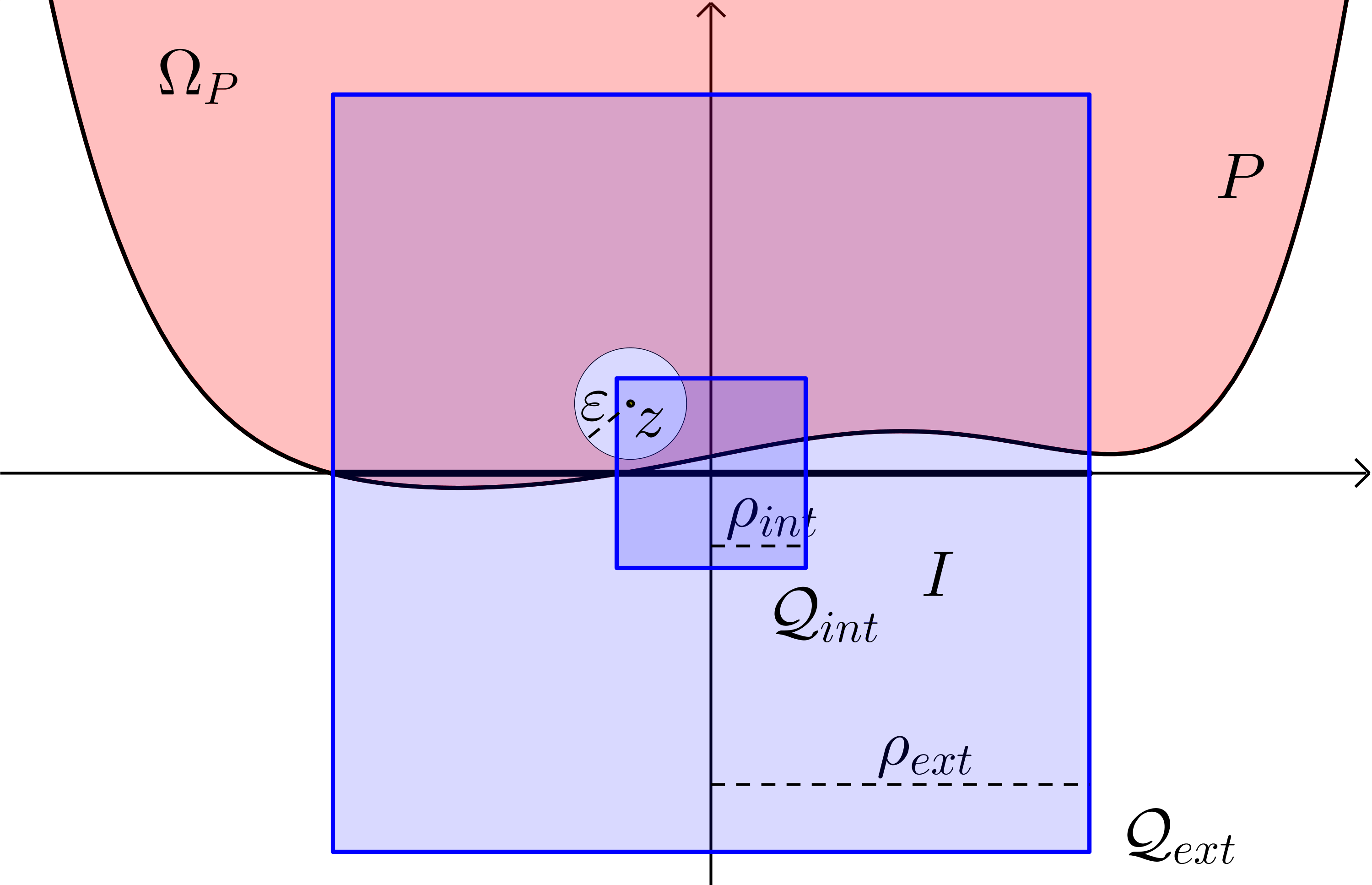}}    
   \caption{Disposition in Proposition \ref{propoPoly}.}\label{figIxI}
\end{figure}

First consider $n=1$. In that case, $\Omega_P$ is a half plane. By rotation and dilation, we can assume $\Omega_P=\R^2_+:=\{w=x+i\,y: \, y>0\}$. Note that $\frac{(\overline{z-w})^{j_2}}{(z-w)^{j_1-1}}$ is infinitely many times differentiable with respect to $w$ in any annulus centered in $z\in\R^2_+$. Then we can apply Green's formula \rf{eqGreen} and use the decay at infinity of the integrand and  \rf{eqCancellation} to see that for $\varepsilon>0$ small enough
\begin{align*}
\int_{\R^2_+\setminus B(z,\varepsilon)}\frac{(\overline{z-w})^{j_1-3}}{(z-w)^{j_1}}\, dm(w)
	&=c_{j_1} \int_\R\frac{(\overline{z-w})^{j_1-3}}{(z-w)^{j_1-1}}\, d\overline w =c_{j_1} \int_\R\frac{(\overline{z-w})^{j_1-3}}{(z-w)^{j_1-1}}\, dw\\
	& =c_{j_1} \int_{\R^2_+\setminus B(z,\varepsilon)}\frac{(\overline{z-w})^{j_1-4}}{(z-w)^{j_1-1}}\, dm(w).
\end{align*}
When $j_1=3$ the last constant is zero. By induction, all these integrals equal zero.

Now we assume that $n\geq2$. Consider a given $\rho_{ext}>0$. We define the interval $I:=[-\rho_{ext},\rho_{ext}]$, the {\em exterior window}  $\mathcal{Q}_{ext}:=Q(0, \rho_{ext})$, and the {\em interior window} $\mathcal{Q}_{int}:=Q(0, \rho_{int})$. Note that \rf{eqBoundCoefficients} implies that for $\rho_{ext}$ small enough, the set $\{x+i\,P(x): x\in I\}\subset \mathcal{Q}_{ext}$, that is, the boundary $\partial\Omega_P$, intersects the vertical sides of the window $\mathcal{Q}_{ext}$ but does not intersect the horizontal ones. The same can be said for the sides of $\mathcal{Q}_{int}$ (see Figure \ref{figIxI}).

Fix $z\in\mathcal{Q}_{int}$ and $\varepsilon<\dist(z, \partial\Omega)$. Splitting the domain of integration in two regions we get
\begin{equation}\label{eqIntegral2Terms}
 \int_{\Omega_P\setminus B(z,\varepsilon)} \frac{(\overline{z-w})^{j_2}}{(z-w)^{j_1}}\, dm(w)= \int_{\Omega_P\setminus \mathcal{Q}_{ext}} \frac{(\overline{z-w})^{j_2}}{(z-w)^{j_1}}\, dm(w)+ \int_{\Omega_P\cap \mathcal{Q}_{ext}\setminus B(z,\varepsilon)} \frac{(\overline{z-w})^{j_2}}{(z-w)^{j_1}}\, dm(w).
\end{equation} 
We bound the non-local part trivially by taking absolute values and using polar coordinates. Choosing $\rho_{int}<\rho_{ext}/2$, we have that
\begin{equation}\label{eqIntegralNonLocal}
 \int_{\Omega_P\setminus \mathcal{Q}_{ext}} \frac{1}{|z-w|^{j_1-j_2}}\, dm(w)\leq \int_{\frac{\rho_{ext}}{2}}^\infty\frac{1}{r^{j_1-j_2}}\int_0^1\,dm_1\,2\pi r\,dr  = \frac{2\pi}{j_1-j_2-2} \frac{2^{j_1-j_2-2}}{(\rho_{ext})^{j_1-j_2-2}},
\end{equation}
where $dm_1$ stands for the Lebesgue length measure. Note that $j_1-j_2-2=n$.

To bound the local part, we can apply Green's Theorem again and we get
\begin{align}\label{eqReduceToBoundary}
\frac{2(j_1-1)}{i} \int_{\Omega_P\cap \mathcal{Q}_{ext}\setminus B(z,\varepsilon)} \frac{(\overline{z-w})^{j_2}}{(z-w)^{j_1}}\, dm(w)
\nonumber	&=  \int_{|z-w|=\varepsilon} \frac{(\overline{z-w})^{j_2}}{(z-w)^{j_1-1}}\, d\overline{w}\\
\nonumber	& \quad +  \int_{\Omega_P\cap\partial\mathcal{Q}_{ext}} \frac{(\overline{z-w})^{j_2}}{(z-w)^{j_1-1}}\, d\overline{w}\\
			& \quad -  \int_{\partial\Omega_P\cap\mathcal{Q}_{ext}} \frac{(\overline{z-w})^{j_2}}{(z-w)^{j_1-1}}\, d\overline{w}.
\end{align}

The first term in the right-hand side of \rf{eqReduceToBoundary} is zero arguing as in \rf{eqCancellation}.
For the second term we note that $z\in\mathcal{Q}_{int}$, and every $w$ in the integration domain is in $\partial \mathcal{Q}_{ext}$, so $|z-w|> \rho_{ext}-\rho_{int}$. Thus, 
\begin{equation}\label{eqIntegralBoundaryQ}
 \int_{\Omega_P\cap\partial\mathcal{Q}_{ext}} \frac{1}{|z-w|^{j_1-j_2-1}}\, d\overline{w}
 	\leq \frac{1}{|\rho_{ext}-\rho_{int}|^{j_1-j_2-1}}6\rho_{ext}.
\end{equation}
Summing up, by \rf{eqIntegral2Terms}, \rf{eqIntegralNonLocal}, \rf{eqReduceToBoundary} and \rf{eqIntegralBoundaryQ}, since $\rho_{int}<\frac{\rho_{ext}}2$,  we get that
\begin{equation}\label{eqIntegralReduced}
\left| \int_{\Omega_P\setminus B(z,\varepsilon)} \frac{(\overline{z-w})^{j_2}}{(z-w)^{j_1}}\, dm(w)\right| 
	\leq \left|\int_{\partial\Omega_P\cap\mathcal{Q}_{ext}} \frac{(\overline{z-w})^{j_2}}{(z-w)^{j_1-1}}\, d\overline{w}\right| + \frac{C_n}{\rho_{ext}^{n}},
\end{equation} 
with $C_n$ depending only on $n$.

It remains to bound the first term in the right-hand side of \rf{eqIntegralReduced}. We begin by using the change of coordinates $w=x+i\,P(x)$ to get a real variable integral:
\begin{equation}\label{eqChangeOfVariables}
\int_{\partial\Omega_P\cap\mathcal{Q}_{ext}} \frac{(\overline{z-w})^{j_2}}{(z-w)^{j_1-1}}\, d\overline{w}=\int_I\frac{(\overline{z}-(x-i\,P(x)))^{j_2}}{(z-(x+i\,P(x)))^{j_1-1}}\,(1-i\,P'(x))\, dx .
\end{equation}
Note that the denominator on the right-hand side never vanishes because $z\notin\partial\Omega_P$. 
Now we take a closer look to the fraction in order to take as much advantage of cancellation as we can, namely
\begin{align}\label{eqChangeOfFractions}
\frac{(\overline{z}-(x-i\,P(x)))^{j_2}}{(z-(x+i\,P(x)))^{j_1-1}}
\nonumber	&=\frac{\Big((\overline{z}-z+2i\,P(x))+(z-(x+i\,P(x)))\Big)^{j_2}}{(z-(x+i\,P(x)))^{j_1-1}}\\
\nonumber	&=\sum_{j=0}^{j_2}{j_2 \choose j}(\overline{z}-z+2i\,P(x))^j(z-(x+i\,P(x)))^{j_2-j-j_1+1}\\
			&=\sum_{j=0}^{j_2}{j_2 \choose j}\frac{(-2 i\,\imag(z)+2i\,P(x))^j}{(z-(x+i\,P(x)))^{n+1+j}}.
\end{align}

%It remains to bound the first term in \rf{eqIntegralReduced}. We begin by using the change of coordinates $w=x+i\,P(x)$ to get a real variable integral
%\begin{equation}\label{eqChangeOfVariables}
%\int_{\partial\Omega_P\cap\mathcal{Q}_{ext}} \frac{(\overline{z-w})^{j_2}}{(z-w)^{j_1-1}}\, d\bar{w}=\int_I\frac{(\bar{z}-(x-i\,P(x)))^{j_2}}{(z-(x+i\,P(x)))^{j_1-1}}\,(1-i\,P'(x))\, dx .
%\end{equation}
%Note that the denominator on the right-hand side never vanishes because $z\notin\partial\Omega_P$. 

Next, we complexify the right-hand side of \rf{eqChangeOfFractions} so that we have a holomorphic function in a certain neighborhood of $I$ to be able to change the  integration path. To do this change we need a key observation. If $\tau\in\mathcal{Q}_{ext}$, then $|\tau|<\sqrt{2}\rho_{ext}$ and by \rf{eqBoundCoefficients}  writing $\widetilde{\delta}=3^n\delta$ we have that
\begin{align}\label{eqBoundPolyDerivative}
|P'(\tau)|
	& \leq |a_1|+2|a_2||\tau|+ \cdots \leq \widetilde{\delta}\left( \frac{ \rho_{int}}{R} + \frac{2}{ R} 2\rho_{ext}+ \frac{3}{ R^2} (2\rho_{ext})^2+\cdots\right) <1/2
\end{align}
if $\rho_{ext}$ is small enough. Thus, we have that $\real(1+i\,P'(\tau))>\frac12$ in $\mathcal{Q}_{ext}$ and, by the Complex Rolle Theorem \ref{theoRolle},  we can conclude that $\tau \mapsto \tau+i\,P(\tau)$ is injective in $\mathcal{Q}_{ext}$. In particular, $z-(\tau+i\,P(\tau))$ has one zero at most in $\mathcal{Q}_{ext}$, and this zero is not real because $z\notin\partial\Omega_P$. Therefore, since the real line divides $\mathcal{Q}_{ext}$ in two congruent open rectangles, there is one of them whose closure has a neighborhood containing no zeros of this function. We call this open rectangle  $\mathcal{R}$. Now,  for any $j\geq 0$ we have that $\tau\mapsto \frac{(P(\tau)-\imag(z))^j}{(z-(\tau+i\,P(\tau)))^{n+1+j}}\,(1-i\,P'(\tau))$ is holomorphic in $\mathcal{R}$, so we can change the path of integration and get
\begin{equation}\label{eqChangePath}
\int_I\frac{2^j(P(x)-\imag(z))^j}{(z-(x+i\,P(x)))^{n+1+j}}\,(1-i\,P'(x))\, dx =- \int_{\partial\mathcal{R}\setminus I}\frac{2^j(P(\tau)- \imag(z))^j}{(z-(\tau+i\,P(\tau)))^{n+1+j}}\,(1-i\,P'(\tau))\, d\tau. 
\end{equation} 

On the other hand, if $|\tau|<\sqrt2\rho_{ext}$, then we have that
\begin{align}\label{eqBoundPolyItself}
|P(\tau)|
			& \leq |a_0|+|a_1||\tau|+|a_2|| \tau |^2+|a_3|| \tau |^3+\cdots\\
\nonumber	& \leq \widetilde{\delta}\left( \frac{\rho_{int}^2}{R} + \frac{\rho_{int}}{R}\, 2\rho_{ext} + \frac{1}{R} (2\rho_{ext})^2+\frac{1}{R^{2}} (2\rho_{ext})^3+\cdots\right)\leq \rho_{ext}^{3/2}
\end{align}
for $\rho_{ext}$ small enough. Then, taking absolute values inside the last integral in \rf{eqChangePath} and using \rf{eqBoundPolyDerivative} and \rf{eqBoundPolyItself} we get
\begin{equation}\label{eqReduceToNothing}
\int_{\partial\mathcal{R}\setminus I}\frac{2^j|P(\tau)-\imag(z)|^{j}}{|z-(\tau+i\,P(\tau))|^{n+1+j}}\,|1-i\,P'(\tau)|\, |d\tau|
\leq \frac{3}{2} \int_{\partial\mathcal{R}\setminus I}\frac{2^j(\rho_{ext}^{3/2}+\rho_{int})^{j}}{|z-(\tau+i\,P(\tau))|^{n+1+j}}\, |d\tau|. 
\end{equation} 

Finally, for any $\tau\in\partial R\setminus I\subset \partial\mathcal{Q}_{ext}$ and $\rho_{ext}$ small enough, we have that 
\begin{equation*}
|z-(\tau+i\,P(\tau))| \geq |\tau|-|z|-|P(\tau)| \geq \rho_{ext} - \sqrt{2}\rho_{int} - \rho_{ext}^\frac32 \geq \frac{\rho_{ext}}{2}-2 \rho_{int}.
\end{equation*}
Using this fact we rewrite \rf{eqReduceToNothing} as
\begin{equation}\label{eqReduceToNothingButNothing}
\int_{\partial\mathcal{R}\setminus I}\frac{2^j |P(\tau)-\imag(z)|^{j}}{|z-(\tau+i\,P(\tau))|^{n+1+j}}\,|1-i\,P'(\tau)|\, |d\tau|
	\leq\frac{3}{2}\frac{2^j(\rho_{ext}^{3/2}+\rho_{int})^{j}}{(\rho_{ext}/2-2\rho_{int})^{n+1+j}} \int_{\partial\mathcal{R}\setminus I}\, |d\tau|. 
\end{equation} 

Putting together \rf{eqChangeOfVariables}, \rf{eqChangeOfFractions}, \rf{eqChangePath} and \rf{eqReduceToNothingButNothing} we can write
\begin{align}
\left|\int_{\partial\Omega_P\cap\mathcal{Q}_{ext}} \frac{(\overline{z-w})^{j_2}}{(z-w)^{j_1-1}}\, d\overline{w}\right|
\nonumber	& \leq \frac{3}{2\, (\rho_{ext}/2-2\rho_{int})^{n+1}}\sum_{j=0}^{j_2} \left(2\cdot\frac{\rho_{ext}^{3/2}+\rho_{int}}{\rho_{ext}/2-2\rho_{int}}\right)^j {j_2 \choose j} \, 4\rho_{ext}\\
\nonumber	& =\frac{6\rho_{ext}}{(\rho_{ext}/2-2\rho_{int})^{n+1}}\left(1+2 \cdot \frac{\rho_{ext}^{3/2}+\rho_{int}}{\rho_{ext}/2-2\rho_{int}}\right)^{j_2}, 
\end{align}
and, choosing $\rho_{int} = \min\{\rho_{ext}/8, \rho_{ext}^{3/2}\}$, 
\begin{equation}\label{eqTakeAbsoluteBound}
\left|\int_{\partial\Omega_P\cap\mathcal{Q}_{ext}} \frac{(\overline{z-w})^{j_2}}{(z-w)^{j_1-1}}\, d\overline{w}\right|
	 \leq \frac{C_n}{\rho_{ext}^{n}}\left(1+16\rho_{ext}^{1/2}\right)^{j_2}, 
\end{equation}
where the constant $C_n$ depends only on $n$.

Now, \rf{eqIntegralReduced} together with \rf{eqTakeAbsoluteBound} prove \rf{eqPolyIntegralBounded}.
\end{proof}

\begin{remark}
Note that we have assumed $\gamma_2\geq 0$ in the proof of Theorem \ref{mtheorem}. When proving the case $\gamma_2\leq 0$, we would have to prove Proposition \ref{propoPoly} with $\gamma \in \{(j_1,-j_2): j_1,j_2\in \N_0 \mbox{ and }  j_2-j_1=n+2\}$. The proof is analogous to the one shown above with slight modifications, and it is left to the reader to complete the details.
\end{remark}

\subsection{Bounded domains: a localization principle}\label{secLocalization}

In this section we use a standard localization procedure to deduce the following result from Theorem \ref{mtheorem}.

\begin{theorem}\label{theoGeometric}
Let  $n \in\N$, $1<p<\infty$, let $\delta,R>0$ and  let $\Omega$ be a bounded $(\delta,R)$-$C^{n-1,1}$ domain with parameterizations in $B^{n+1-1/p}_{p,p}$. Then, for any $\gamma \in \Z^2\setminus \{(-1,-1)\}$ with $\gamma_1+\gamma_2=-2$, we have that $T^\gamma \chi_\Omega \in W^{n,p}(\Omega)$ and, in particular, for any $\epsilon>0$, we have that 
\begin{equation}\label{eqBoundNormTGamma}
\norm{\nabla^n T^\gamma \chi_\Omega}_{L^p(\Omega)}^p\lesssim C_\epsilon |\gamma|^{np} \left(\norm{N}_{B^{n-1/p}_{p,p}(\partial\Omega)}^p+  (1+\epsilon)^{|\gamma|p}\right),
\end{equation}
where $C_\epsilon$ depends on $n$, $p$, $\delta$, $R$, the length of the boundary $\mathcal{H}^1(\partial\Omega)$ and $\epsilon$ but not on $|\gamma|$. 
\end{theorem}

Note that the result above implies Theorem \ref{theoGeometricVeryNaive} as a particular case. 

Along this section, we consider $n \in\N$, $1<p<\infty$, $\delta>0$, $R>0$ to be fixed.  Let $\Omega$  be a $(\delta,R)$-$C^{n-1,1}$ domain. To show that it satisfies \rf{eqBoundNormTGamma} we will find bounds for $\norm{D^\alpha T^\gamma \chi_\Omega}_{L^p(\Omega)}$ below, where  $\alpha \in \N^2_0$ with $|\alpha|=n$. First of all, we need to find out who are the derivatives of $T^\gamma\chi_\Omega$ that we want to estimate. This is particularly important since, in order to use Theorem \ref{mtheorem}, we will substitute $\Omega$ by admissible domains $\widetilde{\Omega}$, which  are unbounded and, therefore, $T^\gamma \chi_{\widetilde{\Omega}}$ is not well-defined for those domains when $\gamma_1+\gamma_2=-2$. We could avoid this problem by defining $T^\gamma$ in BMO, but we will skip those technicalities and substitute $D^\alpha T^\gamma$ by $T^{\gamma-\alpha}$ as our next lemma shows.

\begin{figure}[ht]
 \centering
 {\includegraphics[width=0.7\textwidth]{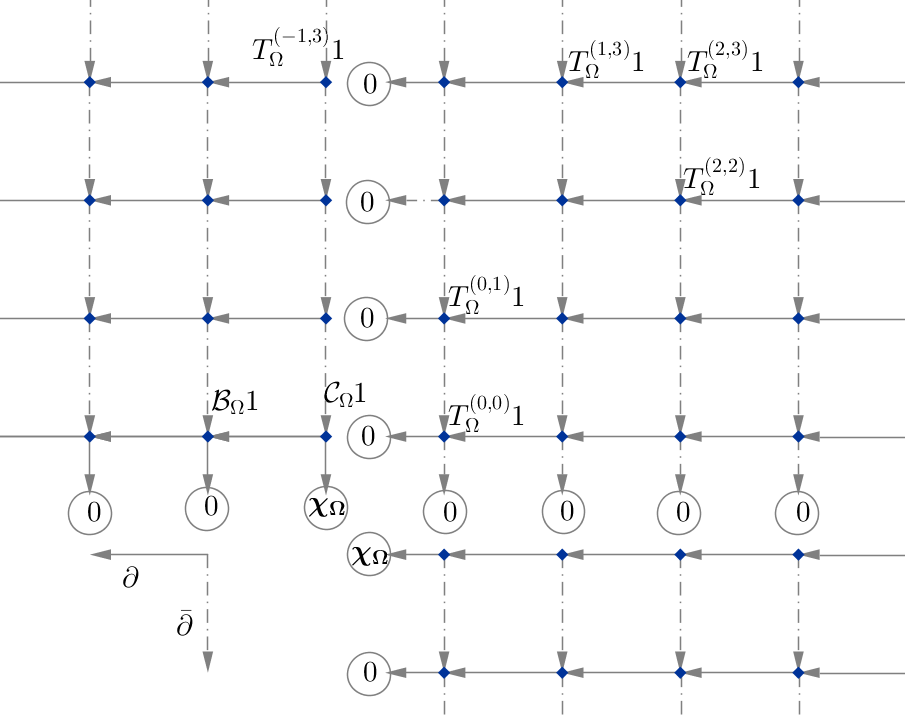}}    
  \caption{This diagram represents $T^{\gamma}_\Omega 1$ (with $\gamma_1$ represented in the horizontal axis and $\gamma_2$ in the vertical one) and the relations found in Lemma \ref{remDiffCharacteristic} between them via weak-derivation in $\Omega$, $\partial$-derivation follows  right-to-left direction, $\overline\partial$-derivation follows top-to-bottom direction. For instance, $D^{(1,3)}T^{(-1,3)}_\Omega 1=\partial \bar{\partial}^3 T^{(-1,3)}_\Omega 1=c\,\Beurling _\Omega 1$. Note that $T^{(0,j)}_\Omega 1$ for $j\neq -1$ are anti-analytic, while  $T^{(j,0)}_\Omega 1$ for $j\neq -1$ are analytic.}\label{figDerivatives}
\end{figure}

\begin{lemma}\label{remDiffCharacteristic}
Consider a bounded $(\delta,R)-C^{(n-1,1)}$ domain $\Omega$ and let us fix $\gamma\in \Z^2$ with either $\gamma_1\geq0$ or $\gamma_2\geq 0$, and $\alpha\in \N_0^2$ with modulus $|\alpha|=n$. Then for $z\in\Omega$ we have
 \begin{equation*}
 D^\alpha T^\gamma_\Omega 1 (z) =
 	\begin{cases} 
		C_n \chi_\Omega (z) & \mbox{if } \gamma=(n-1,-1) \mbox{ and } \alpha=(n,0) \\
		 & \quad \mbox{or } \gamma=(-1,n-1) \mbox{ and } \alpha=(0,n), \\
		0 & \mbox{if } \alpha_1>\gamma_1\geq0 \mbox{ or }\alpha_2>\gamma_2\geq 0 \mbox{ except in the previous case,}\\
		C_{\gamma,\alpha} T^{\gamma-\alpha}_\Omega 1(z) & \mbox{otherwise},
	 \end{cases} 
\end{equation*}
where $D^\alpha$ stands for the weak derivative in $\Omega$ (see Figure \ref{figDerivatives}). The constants satisfy $|C_{\gamma,\alpha}|\lesssim (|\gamma|+n)^n$ and $|C_n|\lesssim n!$.
\end{lemma}

\begin{proof}
Let us assume that $\gamma_2\geq 0$. If $\gamma_1\geq 0$ as well, differentiating a polynomial under the integral sign makes the proof trivial, so we assume $\gamma_1\leq -1$. Recall that we write $w^\gamma=w^{\gamma_1}\overline{w}^{\gamma_2}$. For every $z\in\Omega$ choose $\varepsilon_z:=\dist(z,\partial\Omega)/2$. By \rf{eqCancellationGreen}, Green's formula and \rf{eqCancellation} we get that
\begin{equation}\label{eqGreenToTGamma}
T^\gamma  _\Omega 1(z)=\int_{\Omega\setminus B(z,\varepsilon_z)} (z-w)^{\gamma} \, dm(w)=\frac{i}{2(\gamma_2+1)} \int_{\partial\Omega} (z-w)^{\gamma+(0,1)} \,dw,
\end{equation}
and we can differentiate under the integral sign.

If $\gamma_2\geq\alpha_2$, then we have
\begin{align*}
D^\alpha T^\gamma  _\Omega 1(z)
	& = \frac{i}{2(\gamma_2+1)} (-1)^{\alpha_1}\frac{(\gamma_2+1)!}{(\gamma_2-\alpha_2+1)!}\frac{(-\gamma_1+\alpha_1-1)!}{(-\gamma_1-1)!}\int_{\partial\Omega} (z-w)^{\gamma-\alpha+(0,1)} \,dw.	
\end{align*}
Since $\gamma_2- \alpha_2\geq 0$ and $\gamma_1-\alpha_1<0$, we can apply \rf{eqGreenToTGamma} to  $\gamma-\alpha$ instead of $\gamma$ and, thus,
\begin{equation*}
D^\alpha T^\gamma  _\Omega 1(z)
= (-1)^{\alpha_1}\frac{(\gamma_2)!}{(\gamma_2-\alpha_2)!}\frac{(-\gamma_1+\alpha_1-1)!}{(-\gamma_1-1)!}T^{\gamma-\alpha}_\Omega 1(z).
\end{equation*}

If $\gamma_2+1=\alpha_2$ we must pay special attention. In that case differentiating under the integral sign in \rf{eqGreenToTGamma} we get
\begin{align*}
D^\alpha T^\gamma  _\Omega 1(z)
	& = \frac{i}{2} (-1)^{\alpha_1}\frac{(\gamma_2)!}{(\gamma_2-\alpha_2+1)!}\frac{(-\gamma_1+\alpha_1-1)!}{(-\gamma_1-1)!}\int_{\partial\Omega} (z-w)^{\gamma-\alpha+(0,1)} \,dw\\
	&=C_{\gamma,\alpha}\int_{\partial\Omega}\frac{1}{(z-w)^{-\gamma_1+\alpha_1}} \,dw,		
\end{align*}
where $|C_{\gamma,\alpha}|\leq (|\gamma|+n)^n$.
If, moreover, $\gamma_1-\alpha_1\leq -2$, we can use \rf{eqCancellation} and Green's Theorem to write
\begin{equation}\label{eqBorderLineDerivation}
D^\alpha T^\gamma  _\Omega 1(z)=C_{\gamma,\alpha}\int_{\partial\Omega\cup \partial B(0,\varepsilon_z)}\frac{1}{(z-w)^{-\gamma_1+\alpha_1}} \,dw=C_{\gamma,\alpha}\int_{\Omega\setminus \partial B(0,\varepsilon_z)} 0 \,dm(w)=0.
\end{equation} 
Otherwise, that is, if $\gamma_2+1=\alpha_2$ and $\gamma_1-\alpha_1=-1$, then $\alpha=(0,n)$ and $\gamma=(-1,n-1)$. This implies that
\begin{equation}\label{eqBorderLineDerivationExtra}
D^\alpha T^\gamma  _\Omega 1(z)=C_n \int_{\partial\Omega }\frac{1}{(z-w)} \,dw=C_n \chi_\Omega(z),
\end{equation} 
with $|C_n|\lesssim (n-1)!$.
Let us remark  the fact that $\gamma=(-1,0)$ together with $\alpha=(0,1)$ is the case of the $\bar\partial$-derivative of the Cauchy transform, which is the identity.

Finally, if $\gamma_2<\alpha_2-1$, then differentiating \rf{eqBorderLineDerivation} or \rf{eqBorderLineDerivationExtra} we get
\begin{equation*}
D^\alpha T^\gamma  _\Omega 1(z)=0.
\end{equation*}
One can argue analogously if $\gamma_1\geq0$.
\end{proof}

\begin{figure}[ht]
 \centering
 {\includegraphics[width=0.8\textwidth]{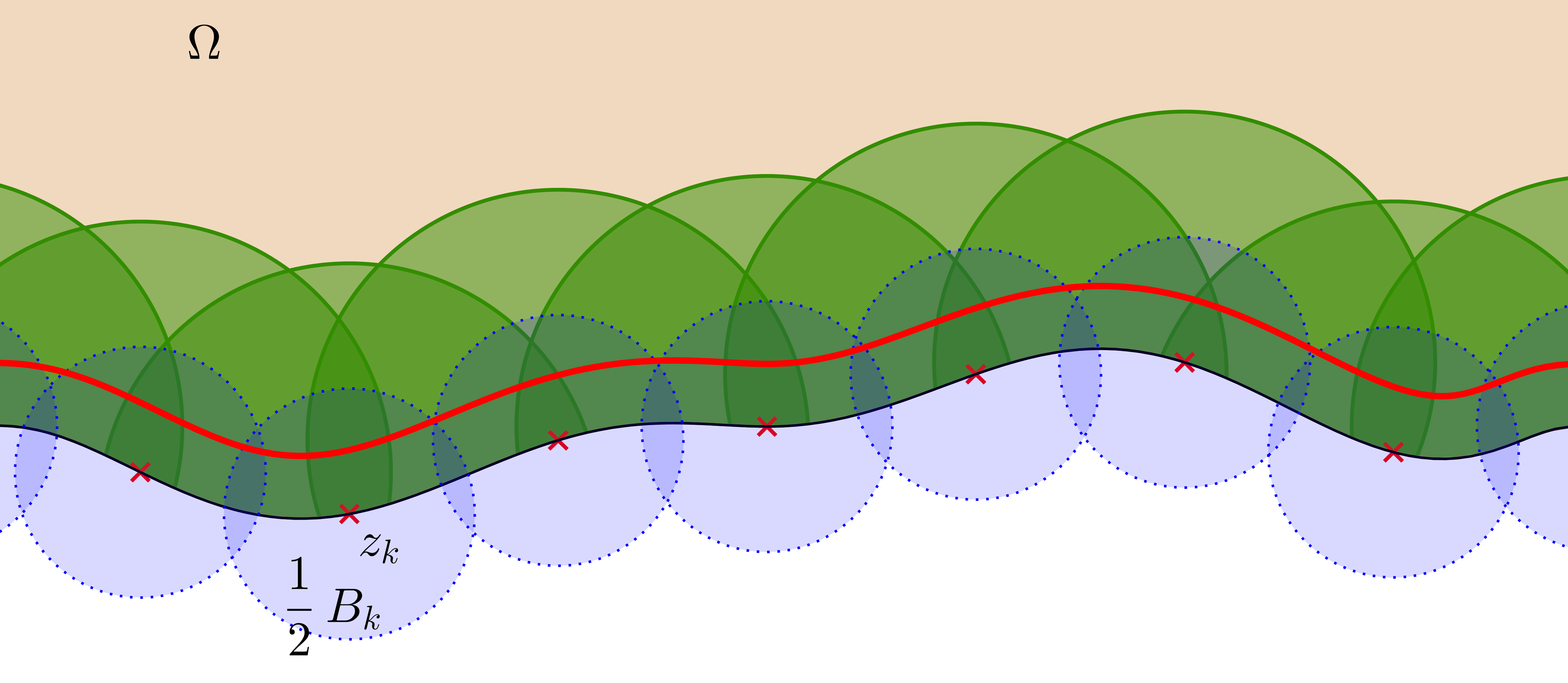}}    
  \caption{Decomposition of $\Omega$ in the proof of Theorem \ref{theoGeometric}. The inner region $\widetilde{\Omega}_0$ appears in the figure above the  bold red line, while the family $\{B_k\cap \Omega\}$ appear in green.}\label{figFinalDecomposition}
\end{figure}

\begin{proof}[Proof of Theorem \ref{theoGeometric}]
Let $\Omega$  be a $(\delta,R)$-$C^{n-1,1}$ domain and let  $\gamma \in \Z^2\setminus \{(-1,-1)\}$ with $\gamma_1+\gamma_2=-2$ and $\alpha \in \N^2$ with $|\alpha|=n$. By Lemma \ref{remDiffCharacteristic}, if $\gamma-\alpha$ has two negative coordinates,  $D^\alpha T^\gamma _\Omega 1$ agrees with a constant  (either null or not bounded by $C_n$) on $\Omega$ and, thus, \rf{eqBoundNormTGamma} follows.

Therefore, we can assume that 
\begin{equation}\label{eqSwitchDerivatives}
D^\alpha T^\gamma _\Omega 1=C_{\gamma,\alpha}T^{\nu} _\Omega 1,
\end{equation}
 with $|C_{\gamma,\alpha}|\lesssim (|\gamma|+n)^n$ and $\nu_1+\nu_2 =-n-2$ with $\nu_1\cdot \nu_2\leq 0$. Let $0<\rho_\epsilon<\frac{R}{20}$ to be chosen as in Theorem \ref{mtheorem}. Let us divide $\Omega$ in several subregions, one of them away from the boundary, say 
$$\widetilde{\Omega}_0:=\left\{z\in\Omega: \dist(z,\partial\Omega)>\frac{\rho_\epsilon}4 \right\},$$
 and the rest being contained in small balls ${B}_k:={B}(z_k,\rho_\epsilon)$, centered in the boundary point $z_k$, with controlled overlapping (namely, we require that the family $\{\frac14 {B}_k\}$ is disjoint while the family  $\{\frac12 B_k\}$ covers $\partial\Omega$, see Figure \ref{figFinalDecomposition}) so that the boundary of $\Omega$ coincides, after rotation and translation, with the boundary of a $(\delta,R,n,p)$-admissible  domain $\Omega_k$ in the strip $(-6\rho_\epsilon, 6\rho_\epsilon)\times\R$ (this is possible by Definitions \ref{defLipschitz} and \ref{defAdmissible}). Then, we have that
\begin{equation}\label{eqBreakInDomains}
\norm{T^{\nu}_\Omega 1}_{L^p(\Omega)}\leq  \norm{T^{\nu}_\Omega 1}_{L^p(\widetilde{\Omega}_0)}+ \sum_k\norm{T^{\nu}_\Omega 1}_{L^p(B_k\cap \Omega)}.
\end{equation}
The term corresponding to the central region is an error term. Namely, for $z\in\widetilde{\Omega}_0$ we have that 
$$|T^{\nu} _\Omega 1 (z)|\leq \int_{|w-z|>{\rho_\epsilon}/5} \frac{1}{|w-z|^{n+2}} \, dm(w)\lesssim   \frac{1}{\rho_\epsilon^n}$$
and, therefore, 
\begin{equation}\label{eqInnerRegion}
  \norm{T^{\nu}_\Omega 1}_{L^p(\widetilde{\Omega}_0)}\lesssim \frac{1}{\rho_\epsilon^n}|\Omega|^\frac1p.
\end{equation}
\begin{figure}[ht]
 \centering
 {\includegraphics[width=0.7\textwidth]{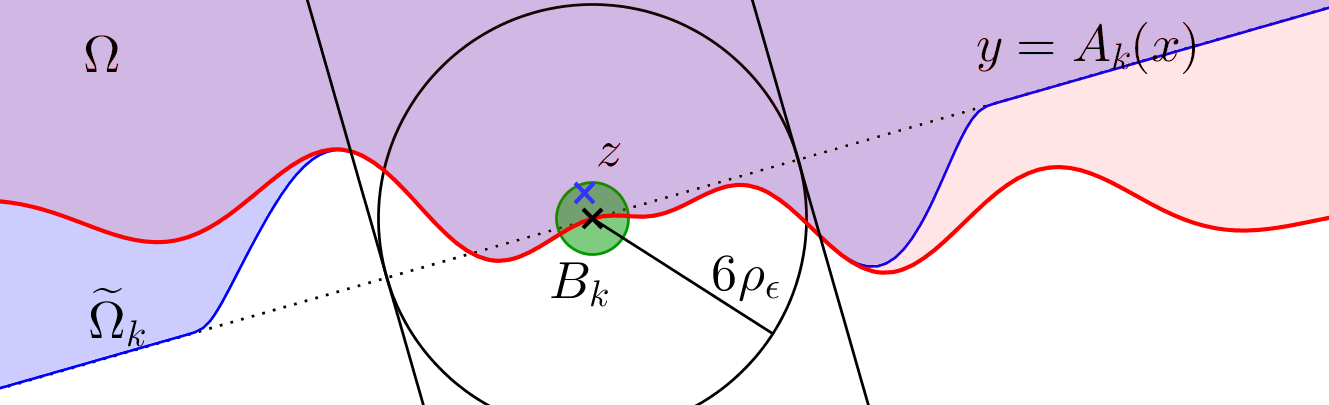}}    
  \caption{Disposition of the domains $\Omega$ and $\widetilde{\Omega}_k$ in the proof of Theorem \ref{theoGeometric} before the rotation and the translation.}\label{figDefiningFunction}
\end{figure}

For the peripheral regions (i.e., close to the boundary of the domain), we use
\begin{equation}\label{eqBreakInDomainsInsideTheIntegral}
\norm{T^{\nu}_\Omega 1}_{L^p(B_k\cap \Omega)}=  \norm{T^{\nu}_\Omega 1}_{L^p(B_k\cap \widetilde{\Omega}_k)}\leq \norm{T^{\nu} \chi_{\widetilde{\Omega}_k}}_{L^p(B_k\cap \widetilde{\Omega}_k)}+\norm{T^{\nu} \left(\chi_{\Omega}-\chi_{\widetilde{\Omega}_k}\right)}_{L^p(B_k\cap \widetilde{\Omega}_k)},
\end{equation}
 where we wrote $\widetilde{\Omega}_k$ for the preimage of ${\Omega}_k$ by the corresponding rigid movement (see Figure \ref{figDefiningFunction}). 
 
Arguing as we did with the central region, we have that
\begin{equation}\label{eqPeripheralRegionEasy}
\norm{T^{\nu} \left(\chi_{\Omega}-\chi_{\widetilde{\Omega}_k}\right)}_{L^p(B_k\cap \widetilde{\Omega}_k)}\lesssim \frac{1}{\rho_\epsilon^n}|B_k\cap \Omega|^\frac1p.
\end{equation}
Finally, for the other term, we use Theorem \ref{mtheorem}. Consider $A_k$ to be the defining function of $\Omega_k$. Then, by Theorem \ref{mtheorem} we have that
\begin{equation}\label{eqPeripheralRegionHard}
\norm{T^{\nu} \chi_{\widetilde{\Omega}_k}}_{L^p(\widetilde{\Omega}_k\cap B_k))}^p\leq C \left( \norm{A_k}^p_{\dot B_{p,p}^{n-1/p+1}(-5\rho_\epsilon,5\rho_\epsilon,)}+  \rho_\epsilon^{2-np}(1+\epsilon)^{|\gamma|p}\right).
\end{equation}

Putting together \rf{eqBreakInDomains} and \rf{eqBreakInDomainsInsideTheIntegral} with \rf{eqInnerRegion}, \rf{eqPeripheralRegionEasy} and \rf{eqPeripheralRegionHard}, we get 
$$\norm{T^{\nu}_\Omega 1}_{L^p(\Omega)}^p\leq C \left( \norm{A_k}^p_{\dot B_{p,p}^{n-1/p+1}(-5\rho_\epsilon,5\rho_\epsilon,)}+  \rho_\epsilon^{2-np}(1+\epsilon)^{|\gamma|p}\right).$$
This fact, together with \rf{eqSwitchDerivatives} and Lemma \ref{lemNormBeta}, shows \rf{eqBoundNormTGamma}. 
\end{proof}

\subsection{The case $p>2$}\label{secPGtr2}
From \cite[Theorem 1.1]{PratsTolsa}, we have the following corollary.
\begin{corollary*}
Let $p>2$, $n\geq 1$, let $\Omega\subset \C$ be a bounded Lipschitz domain and let $\gamma\in \Z^2\setminus{(-1,-1)}$ with $\gamma_1+\gamma_2=-2$. Then the following statements are equivalent:
\begin{enumerate}[a)]
\item The truncated operator $T^\gamma_\Omega$ is bounded in $W^{n,p}(\Omega)$.
\item For every polynomial $P$ of degree at most $n-1$, we have that $T^\gamma_\Omega(P)\in W^{n,p}(\Omega)$.
\end{enumerate}
\end{corollary*}
We will use a quantitative version of this corollary. We state it below without proof. We refer the reader to \cite[pages 2965--2969]{PratsTolsa} for the details. 

Let us fix some notation. Given a multiindex $\lambda\in\N_0^2$, we write $P_\lambda(z)=z^{\lambda_1}\bar{z}^{\lambda_2}$, that is, $P_\lambda(z)=z^\lambda$.
\begin{corollary}\label{coroTPQuantitative}
Let $p>2$, $n\geq 1$, let $\Omega\subset \C$ be a Lipschitz domain and let $\gamma\in \Z^2\setminus{(-1,-1)}$ with $\gamma_1+\gamma_2=-2$. Then
\begin{equation}\label{eqBoundEverythingF}
\norm{\nabla^n T^\gamma_\Omega f}_{L^p(\Omega)}\lesssim_n \left(\norm{T^\gamma}_{L^p\to L^p}+ C_\delta \norm{K_{\gamma}}_{CZ}+\sum_{|\lambda|<n} \norm{\nabla^nT^\gamma_\Omega P_\lambda}_{L^p(\Omega)}\right)\norm{f}_{W^{n,p}(\Omega)} ,
\end{equation} 
where we wrote
\begin{equation*}
\norm{K_{\gamma}}_{CZ}:=\sup_{j\leq n,z\in\C \setminus\{0\}} |\nabla^{j}K_{\gamma}(z)||z|^{j+2}.
\end{equation*}
\end{corollary}

Using Theorem \ref{theoGeometric} and Corollary \ref{coroTPQuantitative}, we will prove the following theorem, which in particular implies Theorem \ref{theoGeometricNaive}.
\begin{theorem}\label{theoGeometricPGtr2}
Consider $p>2$,   $n\geq 1$ and let $\Omega$ be a bounded Lipschitz domain with parameterizations in $B^{n+1-1/p}_{p,p}$. Then, for every $\epsilon>0$ there exists a constant $C_\epsilon$ such that for every multiindex $\gamma\in \Z^2\setminus\{(-1,-1)\}$ with $\gamma_1+\gamma_2\geq -2$, one has 
\begin{equation}\label{eqOperatorsBiggerHom}
\norm{T^\gamma_\Omega}_{W^{n,p}(\Omega)\to W^{n+{\gamma_1+\gamma_2+2},p}(\Omega)}\leq C_\epsilon |\gamma|^{n+{\gamma_1+\gamma_2+2}}\left( \norm{N}_{B^{n-1/p}_{p,p}(\partial\Omega)}+(1+\epsilon)^{|\gamma|}\right) + \diam(\Omega)^{\gamma_1+\gamma_2+2}.
\end{equation}

In particular, for every $m\in\N$ we have that the iteration of the Beurling transform $(\Beurling^m)_\Omega$ is bounded in $W^{n,p}(\Omega)$, with norm
\begin{equation}\label{eqOperatorsBiggerHomBeurling}
\norm{(\Beurling^m)_\Omega}_{W^{n,p}(\Omega)\to W^{n,p}(\Omega)}\leq C_\epsilon m^{n+1}\left( \norm{N}_{B^{n-1/p}_{p,p}(\partial\Omega)}+ (1+\epsilon)^{m}\right).
\end{equation}
\end{theorem}

\begin{proof}
Note that by \rf{eqDifferentialDimension}, we have that $B^{n+1-1/p}_{p,p}\subset B^{n+1-2/p}_{\infty,\infty}$ and, since $1-2/p>0$, we also have that $B^{n+1-2/p}_{\infty,\infty}=C^{n,1-2/p}$ (see \cite[Section 2.5.7]{TriebelTheory}) so $\Omega$ is in fact a $(\delta,R)$-$C^{n-1,1}$-domain, where $\delta$ and $R$ depend on the size of the local parameterizations of the boundary and on $\norm{N}_{B^{n-1/p}_{p,p}(\partial\Omega)}+\mathcal{H}^1(\partial\Omega)$. Therefore, we can use Theorem \ref{theoGeometric}. 

First we study the case ${\gamma_1+\gamma_2+2}=0$. Consider a given $\gamma\in\Z^2\setminus\{(-1,-1)\}$ with $\gamma_1+\gamma_2=-2$. Recall that for $m\neq 0$, $\Beurling^m=\frac{(-1)^m m}{\pi}T^{(-m-1,m-1)}$ by \rf{eqIterateBeurling}. The proof of the $L^p$ boundedness of these operators with norm smaller than $C_p m^2$  can be found in \cite[Corollary 4.5.1]{AstalaIwaniecMartin}.  Thus, for $m=\gamma_2+1=\frac{|\gamma|}2$, we have that
\begin{equation}\label{eqGammaLP}
\norm{T^\gamma}_{L^p\to L^p}=\frac{\pi}{m}\norm{\Beurling^m}_{L^p\to L^p}\lesssim |\gamma|.
\end{equation} On the other hand, a short computation shows that
\begin{equation}\label{eqCZ}
\norm{K_{\gamma}}_{CZ}=\sup_{j\leq n,z\in\C\setminus \{0\}} |\nabla^{j}K_{\gamma}(z)||z|^{j+2}\lesssim |\gamma|^{n},
\end{equation}
with constant depending on $n$.

In order to use Corollary \ref{coroTPQuantitative}, it only remains to check the bounds for $\norm{D^\alpha T^\gamma_\Omega P_\lambda}_{L^p(\Omega)}$ for all multiindices $\alpha, \lambda\in \N_0^2$ with $|\alpha|=n$ and $|\lambda|<n$. 
Using the binomial expansion $w^\lambda=\sum_{\nu\leq\lambda} (-1)^{|\nu|}{\lambda\choose\nu}(z-w)^\nu z^{\lambda-\nu}$, we can write
$$T^\gamma_\Omega P_\lambda(z)= \lim_{\varepsilon\to 0} \int_{\Omega \setminus B_\varepsilon(z)}(z-w)^\gamma w^\lambda\,dm(w)
= \sum_{\vec{0}\leq\nu\leq\lambda}(-1)^{|\nu|}{\lambda\choose\nu} z^{\lambda-\nu}T^{\gamma+\nu}_\Omega 1(z).$$
Differentiating (and assuming that $0\in\Omega$) we find that
$$|\nabla^nT^\gamma_\Omega P_\lambda(z)|\lesssim  2^{n} \sum_{\vec{0}\leq\nu\leq\lambda} \sum_{j=0}^n  (1+\diam(\Omega))^{n}|\nabla^j T^{\gamma+\nu}_\Omega 1(z)|$$
and, thus, by the equivalence of norms in the Sobolev space \rf{eqEquivalenceNormsSobolev}, we have that
\begin{align*}
\norm{\nabla^n T^\gamma_\Omega P_\lambda}_{L^p(\Omega)}^p
	& \lesssim_\Omega \sum_{\vec{0}\leq\nu\leq\lambda} \left(\norm{\nabla^{n+|\nu|} T^{\gamma+\nu}_\Omega 1}_{L^p(\Omega)}^p+\norm{T^{\gamma+\nu}_\Omega 1}_{L^p(\Omega)}^p \right),
\end{align*}
with constants depending on $n$, $p$ and the diameter and the Sobolev embedding constant of $\Omega$. By Lemma \ref{remDiffCharacteristic} and Theorem \ref{theoGeometric}, we have that
\begin{align}\label{eqBoundNablaPoly}
\norm{\nabla^n T^\gamma_\Omega P_\lambda}_{L^p(\Omega)}^p
		& \lesssim \sum_{\gamma\leq \nu\leq\gamma+\lambda}  |\nu|^{np} \left(\norm{N}_{B^{n-1/p}_{p,p}(\partial\Omega)}^p+  (1+\epsilon)^{|\nu| p}\right) + \sum_{\gamma\leq\nu\leq\gamma+\lambda}  \norm{T^{\nu} _\Omega 1}_{L^p(\Omega)}^p.
\end{align}
The Young Inequality \rf{eqYoung} implies that for all functions $f\in L^p$ and $g\in L^1$, $\norm{f \ast g}_{L^p}\leq \norm{f}_{L^p}\norm{g}_{L^1}$. Thus,  for $\gamma<\nu\leq\gamma+\lambda$ we have that
\begin{equation}\label{eqYoungToLp}
\norm{T^{\nu}_\Omega f}_{L^p}\leq\diam(\Omega)^{\nu_1+\nu_2+2}\norm{f}_{L^p},
\end{equation}
 and taking $f=\chi_\Omega$, $\norm{T^\nu _\Omega 1}_{L^p}^p \lesssim 1+ \diam(\Omega)^{(n-1)p+2}$. For $\nu=\gamma$, the same holds  with a slightly worse constant by \rf{eqGammaLP}. Namely, 
 \begin{equation}\label{eqYoungToLpHomo}
\norm{T^{\gamma}_\Omega f}_{L^p}\leq C_p |\gamma| \norm{f}_{L^p}.
\end{equation}

Since $p>2$, putting \rf{eqBoundEverythingF}, \rf{eqGammaLP}, \rf{eqCZ}, \rf{eqBoundNablaPoly} and \rf{eqYoungToLp} and \rf{eqYoungToLpHomo} together, we get
\begin{align}\label{eqControlNablaTGamma}
\norm{\nabla^nT^\gamma_\Omega}_{W^{n,p}(\Omega)\to L^p(\Omega)}
\nonumber	& \lesssim \norm{K_{\gamma}}_{CZ} + \norm{T^\gamma}_{L^p\to L^p}+\sum_{|\lambda|<n}\norm{\nabla^n(T_\Omega P_\lambda)}_{L^p(\Omega)}\\
\nonumber	& \lesssim |\gamma|^{n}+  |\gamma| + |\gamma|^{n} \left(\norm{N}_{B^{n-1/p}_{p,p}(\partial\Omega)}+  (1+\epsilon)^{|\gamma|}\right)\\
	& \lesssim  |\gamma|^{n} \left(\norm{N}_{B^{n-1/p}_{p,p}(\partial\Omega)}+  (1+\epsilon)^{|\gamma| }\right),
\end{align}
with constants depending on $n$, $p$, $\delta$, the diameter of $\Omega$, its Sobolev embedding constant and $\epsilon$, but not on $\gamma$. The estimate \rf{eqControlNablaTGamma}, together with \rf{eqYoungToLpHomo} proves \rf{eqOperatorsBiggerHom} when $\gamma_1+\gamma_2=-2$ and \rf{eqOperatorsBiggerHomBeurling} for every $m>0$.

It remains to study the operators of homogeneity greater than $-2$. In that case we will see that we can differentiate under the integral sign to recover the previous situation. Fix $\gamma\in\Z^2$ such that ${\gamma_1+\gamma_2+2}>0$. By \rf{eqYoungToLp} we have that  $\norm{T^\gamma_\Omega f}_{L^p}\leq\diam(\Omega)^{\gamma_1+\gamma_2+2}\norm{f}_{L^p}$.
Thus, to prove \rf{eqOperatorsBiggerHom}, it suffices to see that for $f\in W^{n,p}(\Omega)$ we have
$$\norm{\nabla^{n+{\gamma_1+\gamma_2+2}} T^\gamma_\Omega f}_{L^p(\Omega)}\leq C_\epsilon |\gamma|^{n+\gamma_1+\gamma_2+2}\left( \norm{N}_{B^{n-1/p}_{p,p}(\partial\Omega)}+(1+\epsilon)^{|\gamma|}\right) \norm{f}_{W^{n,p}(\Omega)}.$$

Since we have shown  \rf{eqOperatorsBiggerHom} for operators with $\gamma_1+\gamma_2+2=0$, it is enough to check that for any $\nu\in\N_0^2$ with $|\nu|={\gamma_1+\gamma_2+2}$ and $z\in\Omega$, we have
 \begin{equation}\label{eqDerivateTgammaF}
 D^\nu T^\gamma_\Omega f(z) =
 	\begin{cases} 
		C_n \chi_\Omega(z) f(z) & \mbox{if } \gamma=(|\nu|-1,-1) \mbox{ and } \nu=(|\nu|,0) \\
		 & \quad \mbox{or } \gamma=(-1,|\nu|-1) \mbox{ and } \nu=(0,|\nu|), \\
		0 & \mbox{if } \nu_1>\gamma_1\geq 0 \mbox{ or }\nu_2>\gamma_2 \geq 0 \mbox{ except in the previous case,}\\
		C_{\nu, \gamma} T^{\gamma-\nu}_\Omega f(z) & \mbox{otherwise},
	 \end{cases} 
\end{equation}
with $|C_n|,|C_{\nu,\gamma}|\lesssim (|\nu|+|\gamma|)^{|\nu|}$.

To prove this statement, take $\alpha\leq \nu-(1,0)$, and note that the partial derivative is
\begin{align*}
\partial T^{\gamma-\alpha}_\Omega f(z)
	& =\frac{\partial_x T^{\gamma-\alpha}_\Omega f(z)-i\partial_y T^{\gamma-\alpha}_\Omega f(z)}{2}\\
	& =\lim_{h \to 0} \frac{T^{\gamma-\alpha}_\Omega (f-f(z)) (z+h)-T^{\gamma-\alpha}_\Omega (f-f(z))(z)}{2h}\\
	& \quad + \lim_{h \to 0} \frac{T^{\gamma-\alpha}_\Omega (f-f(z)) (z+i\, h)-T^{\gamma-\alpha}_\Omega (f-f(z))(z)}{2ih} +\partial T^{\gamma-\alpha}_\Omega 1(z) f(z) \\
	& =:\squared{I}+\squared{II}+\squared{III},
\end{align*}
where $h$ is assumed to be real. Now, the principal value is not needed because ${\gamma_1-\alpha_1+\gamma_2-\alpha_2}>-2$, so
$$\squared{I}=\lim_{h\to 0} \int_{\Omega} \frac{\left((z+h-w)^{\gamma-\alpha}- (z-w)^{\gamma-\alpha}\right)[f(w)-f(z)]}{2h} \, dm(w) .$$
Moreover,  since $f\in C^{0,\sigma}$ for a certain $\sigma>0$ by the Sobolev Embedding Theorem, we get
$$\lim_{h\to 0} \int_{B(z,2|h|)}\frac{\left(|z+h-w|^{\gamma-\alpha} + |z-w|^{\gamma-\alpha}\right)|f(w)-f(z)|}{2h} \, dm(w) =0 .$$
On the other hand, using the Taylor expansion of order two of $(z-w+\cdot)^{\gamma-\alpha}$ around $0$, there exists $\varepsilon=\varepsilon(h,w,z)$ with $|\varepsilon|<h$ such that
$$\squared{I}=\lim_{h\to 0} \int_{\Omega\setminus B(z,2|h|)} \left(\frac{\partial_x (z-w+\cdot)^{\gamma-\alpha}(0)}{2} +\frac{ \partial_x^2 (z-w+\cdot)^{\gamma-\alpha}(\varepsilon) h}{2} \right)(f(w)-f(z)) \, dm(w).$$
Arguing analogously for $\squared{II}$, we get that  
\begin{align*}
\squared{I}+\squared{II}
	& =\lim_{h\to 0} \int_{\Omega\setminus B(z,2|h|)} (\gamma_1-\alpha_1)(z-w)^{\gamma-\alpha-(1-0)} f(w) \, dm(w)\\
	&\quad - \lim_{h\to 0} \int_{\Omega\setminus B(z,2|h|)} (\gamma_1-\alpha_1)(z-w)^{\gamma-\alpha-(1-0)} \, dm(w) f(z)
\end{align*}
 (when taking limits, the Taylor remainder vanishes by the H\"older continuity of $f$).
If $\gamma_1-\alpha_1=0$ then this part is null and $\squared{III}$ will be also null unless $\gamma_2-\alpha_2=-1$ by Lemma \ref{remDiffCharacteristic}. Otherwise, the last term coincides with $\squared{III}$ and they cancel out. By induction, we get \rf{eqDerivateTgammaF}. 
\end{proof}

\begin{appendices}
\section{Appendix}
We prove the following:
\begin{lemma}
Let $n\geq 1$, $\delta, R>0$,  let $\Omega$ be a bounded $(\delta,R)-C^{n-1,1}$ domain and let $\{\mathcal{Q}_k\}_{k=1}^M$ be a collection of $R$-windows  such that $\left\{\frac{1}{20} \mathcal{Q}_k\right\}_{k}$ cover the boundary of $\Omega$ and $\left\{\frac{1}{40}  \mathcal{Q}_k\right\}_{k}$ are disjoint. Let $\{A_k\}_k$ be the parameterizations of the boundary associated to each window. Then, for any $1<p<\infty$
\begin{equation}\label{eqBoundedNormBoundary}
\sum_{k=1}^M\sum_{I\in\mathcal{D}: I\subset \frac16 I_{R}}\frac{\beta_{(n)}( A_k,I)^p}{\ell(I)^{n\,p-2}} \lesssim \sum_{k=1}^M\norm{A_k}_{\dot B^{n+1-1/p}_{p,p}(\frac13 I_{R})}^p \lesssim  \norm{N}_{B^{n-1/p}_{p,p}(\partial\Omega)}^p,
\end{equation}
where $I_R$ stands for the interval $(-R,R)$. The constants depend on $n$, $p$, $\delta$, $R$ and the length of the boundary $\mathcal{H}^1(\partial\Omega)$.
\end{lemma}

Note that $M\approx \frac{\mathcal{H}^1(\partial\Omega)}{R}$.

\begin{proof}
By \rf{eqNormBetas} the first estimate in \rf{eqBoundedNormBoundary} is immediate.

%while the second one is a consequence of Claim \ref{claimnormadelvectornormal2} below  and \rf{eqEquivalentNormBoundary} is a consequence of all this facts and Claim \ref{claimLocalChartsNorm}.

Let us write $s:=n-\frac1p$ and $\{s\}:=1-\frac1p$. Given $t\in\R$, we write $I_t$ for the interval $t I_R$. To prove the second estimate in \rf{eqBoundedNormBoundary}, using the expression \rf{eqBesovDiff} to express the Besov norm in terms of differences together with the fact that 
$$\norm{A_k}^p_{\dot B^{s+1}_{p,p}(I_{1/3})}\approx \norm{A_k^{(n)}}^p_{\dot B^{\{s\}}_{p,p}(I_{1/3})}$$ 
(that is, the so-called lifting property, see  \cite[Theorem 5.2.3/1]{TriebelTheory}) and using an appropriate cut-off function $\chi_{I_{1/3}}\leq\varphi\leq\chi_{I_{5/12}}$  we get $ \norm{A_k^{(n)}}_{\dot B^{\{s\}}_{p,p}(I_{1/3})}\leq  \norm{\varphi A_k^{(n)}}_{\dot B^{\{s\}}_{p,p}}$, so
\begin{align}\label{eqBesovNormA}
\norm{ A_k}^p_{\dot B^{s+1}_{p,p}(I_{1/3})}
\nonumber	& \lesssim\int_{I_{1/2}} \int_{I_{1/2}}  |\varphi(y)|^p \frac{\left|A_k^{(n)}(x)-A_k^{(n)}(y)\right|^p}{|x-y|^{\{s\}p+1}} dydx \\
\nonumber	& +	\int_{I_{1/2}} \int_{I_{1/2}}  \left|A_k^{(n)}(x)\right|^p \frac{|\varphi(x)-\varphi(y)|^p}{|x-y|^{\{s\}p+1}} dydx + 2\int_{I_{5/12}}\int_{I_{1/2}^c} \frac{ |A_k^{(n)}(x)|}{|x-y|^{\{s\}p+1}} dy dx\\
			& \lesssim \int_{I_{1/2}} \int_{I_{1/2}} \frac{\left|A_k^{(n)}(x)-A_k^{(n)}(y)\right|^p}{|x-y|^{\{s\}p+1}} dy \, dx + 1.
\end{align}
 Note that the error terms are absorbed by an additive constant which depends on the $C^{n-1,1}$ constants of the parameterization $A_k$, that is, on $\delta$ and $R$,  uniformly bounded by hypothesis. Next, using \rf{eqHomogeneous}, the lifting property again and some computations, one can express the norm of the normal vector as
 \begin{align}\label{eqReductionBoundaryNorm}
\norm{N}_{B^s_{p,p}(\partial\Omega)}^p 
	& \approx 1 + \int_I \int_{2I} \frac{|\Delta_h (N\circ z)^{(n-1)} (t)|^p}{|h|^{\{s\}p}}  \frac{dh}{|h|}\, dt,
\end{align}
where $I$ is the interval of length $\mathcal{H}^1(\partial\Omega)$ centered at the origin.

Finally, to compare \rf{eqBesovNormA} and \rf{eqReductionBoundaryNorm}, we will use  the functions
$$N_k(x):=\frac{1}{\sqrt{1+A_k'(x)^2}} (A_k'(x),-1)$$
(that is, for each $k$ we take the normal vector to the graph of the $k$-th parameterization of the boundary at $(x, A_k(x))$), to make an intermediate step. Namely, we will show that
\begin{equation}\label{eqTwoSteps}
\sum_{k=1}^M\norm{ A_k}^p_{\dot B^{s+1}_{p,p}(I_{1/3})}\lesssim \sum_{k=1}^M \int_{I_{1/2}} \int_{I_{1/2}} \frac{\left|N_k^{(n-1)}(x)-N_k^{(n-1)}(y)\right|^p}{|x-y|^{\{s\}p+1}} dy \, dx + 1 \lesssim\norm{N}^p_{B^{s}_{p,p}(\partial\Omega)}.
\end{equation}

We begin by the first inequality. Let us write fix a window $\mathcal{Q}_k$. By \rf{eqBesovNormA}, it only remains to check that 
\begin{equation*}
\circled{I}:=\int_{I_{1/2}} \int_{I_{1/2}-x} \frac{|\Delta_h A_k^{(n)}(x)|^p}{|h|^{\{s\}p}} \frac{dh}{|h|} \, dx\lesssim \int_{I_{1/2}} \int_{I_{1/2}-x} \frac{\left|\Delta_h N_k^{(n-1)}(x)\right|^p}{|h|^{\{s\}p}} \frac{dh}{|h|} \, dx + 1.
\end{equation*}
To do so, we need to relate $\Delta_h A_k^{(n)}(x)$ and $\Delta_h N_k^{(n-1)}(x)$. We can write
$$N_k(x)=(N_{k,1}(x), N_{k,2}(x))=g_k(x)(A_k'(x),-1),$$
with
\begin{align*}
g_k(x)		& =\frac{1}{\sqrt{1+A_k'(x)^2}} \mbox{ and, thus, }\\
g_k'(x)		& =-\frac{A_k''(x)A_k'(x)}{\sqrt{1+A_k'(x)^2}^3}=-A_k''(x)A_k'(x)g_k(x)^3,\\
& \cdots.
\end{align*}
First we note the trivial pointwise bounds of the derivatives of $g_k$. The first two bounds are obvious and the rest of them can be deduced by induction, 
\begin{align*}
\nonumber |g_k(x)|	& =\left|\frac{1}{\sqrt{1+A_k'(x)^2}} \right|\leq 1,\\
\nonumber |g_k'(x)|	& =\left|A_k''(x)A_k'(x)g_k(x)^3\right|\leq \frac{\delta^2}{R},\\
\nonumber & \cdots,\\
|g_k^{(j)}(x)|	& \leq \frac{C_\delta}{R^j} \mbox{ for all $j<n$.}
\end{align*}
Analogously, we have similar bounds for the multiplicative inverse of $g_k$, $\widetilde{g}_k=\frac{1}{g_k}$,  
\begin{align*}
\nonumber \left|\widetilde{g}_k(x)\right|	& \leq \sqrt{1+\delta^2} ,\\
\nonumber \left|\widetilde{g}_k'(x)\right|	& = |g_k(x)A_k'(x)A_k''(x)| \leq \frac{\delta^2}{R},\\
\nonumber & \cdots,\\
\left|\widetilde{g}_k^{(j)}(x)\right|	& \leq \frac{C_\delta}{R^j}  \mbox{ for every $j<n$.}
\end{align*}
Thus, for the $k$-th window normal vector
\begin{align*}
|N_{k,2}^{(j)}(x)|	& =\left|g_k^{(j)}(x)\right|\leq\frac{C_\delta}{R^j}  \mbox{\quad\quad for all $j<n$ and}\\
|N_{k,1}^{(j)}(x)|	& =\left|\sum_{i=0}^{j}{j\choose i} A_k^{(i+1)}(x)g_k^{(j-i)}(x)\right| \lesssim_{\delta,j} \sum_{i=0}^j \frac{1}{R^{i}}\frac{1}{R^{j-i}}\approx \frac{1}{R^j}  \mbox{\quad\quad for all $j<n$.}
\end{align*}
Summing up, we have that
\begin{equation}\label{eqcotesdirectes}
\norm{A_k^{(j+1)}}_{L^\infty},\norm{g_k^{(j)}}_{L^\infty},\norm{\widetilde{g}_k^{(j)}}_{L^\infty},\norm{N_k^{(j)}}_{L^\infty} \lesssim_{\delta,n}\frac{1}{R^j}\mbox{\quad\quad for $j<n$.}
\end{equation}
Therefore, using the Mean Value Theorem one gets
\begin{equation}\label{eqcotesdelta}
|\Delta_h A_k^{(j)}(x)|,|\Delta_hg_k^{(j-1)}(x)|, |\Delta_h\widetilde{g}_k^{(j-1)}(x)|, |\Delta_hN_k^{(j-1)}(x)|\lesssim \frac{|h|}{R^j}\mbox{\quad\quad for $j<n$.}
\end{equation}

Now we want to control $|\Delta_h A_k^{(n)}(x)|$ by an expression in terms of the differences of the derivatives of the normal vector, with $x, x+h\in I_{1/2}$.
We have that
\begin{align*}
N_{k,1}^{(n-1)}(x)=\sum_{i=0}^{n-1} {n-1 \choose i }A_k^{(i+1)}(x)g_k^{(n-1-i)}(x).
\end{align*}
Thus, solving for $A_k^{(n)}(x)$ we get 
$$A_k^{(n)}(x)=\frac{N_{k,1}^{(n-1)}(x)-\sum_{i=0}^{n-2} {n-1 \choose i}A_k^{(i+1)}(x)g_k^{(n-1-i)}(x)}{g_k(x)},$$
and taking differences
\begin{align}\label{eqacodtaDeltaa}
|\Delta_h A_k^{(n)}(x)|
	& \lesssim \left|\Delta_h\left(N_{k,1}^{(n-1)}\widetilde{g}_k\right)(x)\right|+\sum_{i=0}^{n-2} \left|\Delta_h\left(A_k^{(i+1)}g_k^{(n-1-i)}\widetilde{g}_k\right)(x)\right|.
\end{align}

On one hand, using \rf{eqcotesdirectes} and \rf{eqcotesdelta} we have that
\begin{align*}
\left|\Delta_h\left(N_{k,1}^{(n-1)}\widetilde{g}_k\right)(x)\right|
	&\leq \norm{\widetilde{g}_k}_{L^\infty}\left|\Delta_h N_{k,1}^{(n-1)} (x)\right|+\norm{N_{k,1}^{(n-1)}}_{L^\infty}\left|\Delta_h{\widetilde{g}_k}(x)\right|\\
	&\lesssim \left|\Delta_h N_{k,1}^{(n-1)} (x)\right|+\frac{1}{R^{n-1}}\frac{|h|}{R}.
\end{align*}
On the other hand, if we consider $0<i\leq n-2$, we obtain analogously
\begin{align*}
\left|\Delta_h\left(A_k^{(i+1)}g_k^{(n-1-i)}\widetilde{g}_k\right)(x)\right|
	&\lesssim \frac{1}{R^{n-1-i}} \left|\Delta_h A_k^{(i+1)}(x)\right|+\frac{1}{R^{i}}\left|\Delta_h g_k^{(n-1-i)}(x)\right|+\frac{1}{R^{n-1}}\left|\Delta_h{\widetilde{g}_k}(x)\right|\\
	& \lesssim \frac{1}{R^{n-1-i}} \frac{|h|}{R^{i+1}}+\frac{1}{R^{i}}\frac{|h|}{R^{(n-i)}}+\frac{1}{R^{n-1}}\frac{|h|}{R}.
\end{align*}
When $i=0$, instead, using that $N_{k,2}^{(n-1)}(x)=-g_k^{(n-1)}(x)$, we obtain that
\begin{align*}
\left|\Delta_h\left(A_k' g_k^{(n-1)}\widetilde{g}_k\right)(x)\right|
	&\lesssim \frac{1}{R^{n-1}} \left|\Delta_h A_k'(x)\right|+\left|\Delta_h g_k^{(n-1)}(x)\right|+\frac{1}{R^{n-1}}\left|\Delta_h{\widetilde{g}_k}(x)\right|\\
	& \lesssim \frac{1}{R^{n-1}} \frac{|h|}{R}+\left|\Delta_h N_{k,2}^{(n-1)}(x)\right| + \frac{1}{R^{n-1}} \frac{|h|}{R}.
\end{align*}
Back to \rf{eqacodtaDeltaa}, we have deduced that
\begin{align*}
|\Delta_h A_k^{(n)}(x)|
	& \lesssim \left|\Delta_h N_k^{(n-1)}(x)\right| +\frac{|h|}{R^{n}}.
\end{align*}

Applying this result,  we obtain that
\begin{align}\label{eqcotaII}
\nonumber \circled{I}
	& \lesssim  \int_{I_{1/2}} \int_{I_{1/2}-x} \frac{\left| \Delta_h N_k^{(n-1)}(x)\right|^p}{|h|^{\{s\}p+1}} dh \, dx + \frac{1}{R^{np}} \int_{-R}^R \frac{\left| h\right|^p}{|h|^{\{s\}p+1}} dh  \int_{I_{1/2}} dx \\
	& \lesssim \int_{I_{1/2}} \int_{I_{1/2}} \frac{\left| N_k^{(n-1)}(x)-N_k^{(n-1)}(y)\right|^p}{|x-y|^{\{s\}p+1}} dy \, dx + R^{1-sp}.
\end{align}
The first inequality in \rf{eqTwoSteps} is obtained summing in $k$.

To prove the second one, note that $t=\tau_k(x)=\int_0^x \widetilde{g}_k$ is the arc parameter of the curve, since
$$\frac{dx}{dt}=\frac{1}{\widetilde{g}_k(x)}=\frac{1}{\sqrt{1+A_k'(x)^2}}.$$
Thus, we have that $\widetilde{N}_k(t):=N_k(\tau_k^{-1}(t))$ is the normal vector (to the graph of the $k$-th parameterization) parameterized by the arc. Of course, we have that $N_k(x)=\widetilde{N}_k(\tau_k(x))$. Therefore,
$$N_k'(x)=\widetilde{N}_k'(\tau_k(x))\tau_k'(x)=\widetilde{N}_k'(\tau_k(x))\widetilde{g}_k(x)$$
and, by induction, for $j\leq n-1$ we get
\begin{equation}\label{eqdescomponN0}
N_k^{(j)}(x)=\sum_{i=1}^{j}\widetilde{N}_k^{(i)}(\tau_k(x))\sum_{\substack{\alpha\in \N^{i}\\|\alpha|=j-i}}C_\alpha \prod_{l=1}^{i}\widetilde{g}_k^{(\alpha_l)}( x).
\end{equation}
 Solving this equation for $\widetilde{N}_k^{(j)}$ and using \rf{eqcotesdirectes}, for $j\leq n-1$
we have that
\begin{equation}\label{eqcotadirectaN}
\norm{\widetilde{N}_k^{(j)}}_{L^\infty(\tau_k(I_R))} \leq  \frac{1}{R^j}.
\end{equation}
Taking  $t=\tau_k(x)$ and $\widetilde{h}=\tau_k(y)-\tau_k(x)$, and applying \rf{eqdescomponN0}, we get
\begin{align*}
|N_k^{(n-1)}(y)-N_k^{(n-1)}(x)|
	&  \leq |\Delta_{\widetilde{h}} \widetilde{N}_k^{(n-1)}(t)| \norm{\widetilde{g}_k}_{L^\infty}^{n-1}\\
	& \quad +\sum_{j=1}^{n-2}|\Delta_{\widetilde{h}} \widetilde{N}_k^{(j)}(t)| \sum_{\substack{\alpha\in \N^{j}\\|\alpha|=n-1-j}}C_\alpha \prod_{i=1}^{j}\norm{\widetilde{g}_k^{(\alpha_i)}}_{L^\infty}\\
	& \quad + \sum_{j=1}^{n-1}\norm{\widetilde{N}_k^{(j)}}_{L^\infty} \sum_{\substack{\alpha\in \N^{j}\\|\alpha|=n-1-j}}C_\alpha  \sum_{i=1}^{j} \prod_{l \neq i} \left|\widetilde{g}_k^{(\alpha_i)}(x)-\widetilde{g}_k^{(\alpha_i)}(y)\right| \norm{\widetilde{g}_k^{(\alpha_l)}}_{L^\infty}.
\end{align*}
Using \rf{eqcotesdirectes}, \rf{eqcotesdelta} and \rf{eqcotadirectaN} we get
\begin{align*}
 |\Delta_{\widetilde{h}} \widetilde{N}_k^{(n-1)}(t)| \norm{\widetilde{g}_k}_{L^\infty}^{n-1}
	 \lesssim |\Delta_{\widetilde{h}} \widetilde{N}_k^{(n-1)}(t)| ,
\end{align*}
for all $j \leq n-2$ and $|\alpha|=n-1-j$ we get
\begin{align*}
|\Delta_{\widetilde{h}} \widetilde{N}_k^{(j)}(t)| \prod_{i=1}^{j}\norm{\widetilde{g}_k^{(\alpha_i)}}_{L^\infty}
	 \lesssim |{\widetilde{h}}| \norm{\widetilde{N}_k^{(j+1)}}_{L^\infty} \prod_{i=1}^{j} \frac{1}{R^{\alpha_i}} \lesssim \frac{|{\widetilde{h}}|}{R^{j+1+|\alpha|}}=\frac{|{\widetilde{h}}|}{R^n}
\end{align*}
and, for all $ j \leq n-1, |\alpha|=n-1-j$, we get
\begin{align*}
& \norm{\widetilde{N}_k^{(j)}}_{L^\infty} \sum_{i=1}^{j} \prod_{l \neq i} |\widetilde{g}_k^{(\alpha_i)}(x)-\widetilde{g}_k^{(\alpha_i)}(y)| \norm{\widetilde{g}_k^{(\alpha_l)}}_{L^\infty}
	 \lesssim \frac{1}{R^j}\frac{|x-y|}{R^{\alpha_i+1}}\frac{1}{R^{|\alpha|-\alpha_i}}\approx\frac{|{\widetilde{h}}|}{R^n}.
\end{align*}
Thus, 
$$|N_k^{(n-1)}(x)-N_k^{(n-1)}(y)|\lesssim  |\Delta_{\widetilde{h}} \widetilde{N}_k^{(n-1)}(t)|+\frac{|{\widetilde{h}}|}{R^n}.$$ 
Therefore, using the bilipschitz change of variables $t=\tau_k(x)$ and $\widetilde{h}=\tau_k(y)-\tau_k(x)$ in \rf{eqcotaII}, we have that
\begin{align}\label{eqcotaIIn}
\nonumber\circled{I} 
	& \lesssim \int_{I_{1/2}} \int_{I_{1/2}} \frac{\left| N_k^{(n-1)}(x)-N_k^{(n-1)}(y)\right|^p}{|x-y|^{\{s\}p+1}} dy \, dx+ R^{1-sp} \\
	& \lesssim \int_{\tau_k(I_{1/2})}  \int_{\tau_k(I_{1/2})-t} \left(\frac{ |\Delta_{\widetilde{h}} \widetilde{N}_k^{(n-1)}(t)|^p}{|{\widetilde{h}}|^{\{s\}p+1}}+\frac{|{\widetilde{h}}|^p}{R^{np}|{\widetilde{h}}|^{\{s\}p+1}}\right) d{\widetilde{h}}\, dt+ R^{1-sp}.
\end{align}
Taking sums on $1\leq k\leq M$ and using \rf{eqBesovNormA}, \rf{eqcotaII} and \rf{eqcotaIIn} we get
\begin{align*}
 \sum_{k=1}^M\norm{ A_k}^p_{\dot B^{s+1}_{p,p}(I_{1/3})}
 	&  \lesssim  \sum_{k=1}^M \left( \int_{\tau_k(I_{1/2})}  \int_{\tau_k(I_{1/2})-t} \frac{ |\Delta_{\widetilde{h}} \widetilde{N}_k^{(n-1)}(t)|^p}{|{\widetilde{h}}|^{\{s\}p+1}} d{\widetilde{h}}\, dt + R^{1-sp} \right) + R^{1-sp}.
\end{align*}
According to our definitions, for each $k$ and $t\in \tau_k(I_{1/2})$ we have that $\widetilde{N}_k(t)$ coincides with a fixed rotation of $N\circ z(t+z^{-1}(z_k))$ where $z^{-1}(z_k)$ is assumed to be chosen in $I$. That is, $\widetilde{N}_k$ coincides with a fixed rotation of $N:\partial\Omega\to S^1$ parametrized by the arc $z:2I\to \partial\Omega$ for values close to $z^{-1}(z_k)$ and pre-composed with a translation.  Namely, $\widetilde{N}_k^{(n-1)}(t)=e^{it_k}(N\circ z)^{(n-1)}(t+z^{-1}(z_k))$ and
\begin{align*}
\sum_{k=1}^M\norm{ A_k}^p_{\dot B^{s+1}_{p,p}(I_{1/3})}
	&  \lesssim  \sum_{k=1}^M  \int_{\tau_k(I_{1/2})}  \int_{\tau_k(I_{1/2})-t} \frac{ |\Delta_{\widetilde{h}}(N\circ z)^{(n-1)}(t+z^{-1}(z_k))|^p}{|{\widetilde{h}}|^{\{s\}p+1}} d{\widetilde{h}}\, dt  + M R^{1-sp}.
\end{align*}
Changing variables, we get
\begin{align*}
\sum_{k=1}^M\norm{ A_k}^p_{\dot B^{s+1}_{p,p}(I_{1/3})}
	& \lesssim 1 + \int_I \int_{2I} \frac{|\Delta_h (N\circ z)^{(n-1)} (t)|^p}{|h|^{\{s\}p}}  \frac{dh}{|h|}\, dt\approx \norm{N}_{B^s_{p,p}(\partial\Omega)}^p .
\end{align*}
\end{proof}

\begin{remark}
Arguing analogously one can show that
$$\norm{N}^p_{B^{s}_{p,p}(\partial\Omega)} \lesssim \sum_{k=1}^N \norm{ {A}_k}^p_{\dot B^{s+1}_{p,p}(I_1)}+1.$$
By \rf{eqTwoSteps} and \cite[Theorem 3]{Marschall}, we have that this condition is equivalent to $N$ being in the trace space of $W^{n,p}(\Omega)$. 
%
%Moreover, let $\mathcal{Q}_1,\mathcal{Q}_2$  be two windows of the same size of a Lipschitz domain and consider parameterizations $A_1: I_1\to \R$, $A_2:I_2 \to \R$ such that there exist rigid transformations $F_1:\mathcal{Q}_1\to\C$ and $F_2:\mathcal{Q}_2\to \C$ such that 
%$$F_1^{-1}\left(\{(x,A_1(x)): x \in I_1\}\right)=F_2^{-1}\left(\{(x,A_2(x)): x \in I_2\}\right).$$
%Assume further that $\ell(I_1)\approx\ell(I_2)\approx R$. For $s<n+1$ we have that
%$$\norm{ A_1}_{\dot B^{s}_{p,p}(\frac12 I_1)}^p + 1\approx \norm{ A_2}_{\dot B^{s}_{p,p}(\frac12 I_2)}^p+1.$$
\end{remark}

\end{appendices}

\renewcommand{\abstractname}{Acknowledgements}
\begin{abstract}
The present work was developed during the author's doctoral studies under the tuition of Xavier Tolsa. 
The author was funded by the European Research
Council under the European Union's Seventh Framework Programme (FP7/2007-2013) /
ERC Grant agreement 320501. Also, partially supported by grants 2014-SGR-75 (Generalitat de Catalunya), MTM-2010-16232 and MTM-2013-44304-P (Spanish government) and by a FI-DGR grant from the Generalitat de Catalunya, (2014FI-B2 00107).\end{abstract}

\bibliography{../../../bibtex/Llibres}
\end{document}